\documentclass[12pt]{amsart}
\usepackage{amssymb,latexsym,amsmath,amscd,mathrsfs,yfonts,hyperref}
\usepackage{fullpage}
\usepackage{eucal}

\input xy
\xyoption{all}

\newcommand{\ZZ}{\mathbb Z}
\newcommand{\QQ}{\mathbb Q}
\newcommand{\CC}{\mathbb C}

\newcommand{\NN}{\mathbb N}

\newcommand{\cpt}{\mathbb K}

\def\bK{\mathbf K}

\def\cO{\mathcal O}
\def\cM{\mathcal M}

\def\nc{\mathrm{nc}}

\def\Perf{\mathtt{Perf}}

\def\dg{\mathrm{dg}}
\def\Ch{\mathtt{Ch}}

\def\C{\mathrm{C}}

\def\KKcat{\mathtt{KK}}

\def\colim{\mathrm{colim}}

\def\id{\mathrm{id}}

\def\st{{\sf stab}}
\def\St{{\sf Stab}}

\def\KQ{\mathrm{KQ}}

\def\SSet{\mathcal{S}\mathtt{et}_\Delta}

\def\ker{\mathrm{ker}}

\def\Mod{\mathtt{Mod}}

\def\op{\mathrm{op}}

\def\prot{\hat{\otimes}}

\def\Fun{\mathrm{Fun}}

\def\NSH{\mathtt{NSH}}

\def\K{\mathrm{K}}

\def\S{\mathcal{S}}

\def\iNSp{\mathtt{NSp}}
\def\pNSp{\mathtt{NSp'}}
\def\iNS{\mathtt{N}\mathcal{S_*}}

\def\ikk{\mathtt{KK_\infty}}
\def\hNSp{\mathtt{hNSp}}
\def\kc{\mathrm{k}^{\mathrm{cn}}}
\def\Ecn{\mathtt{E^{cn}_\infty}}
\def\hEcn{\mathtt{hE^{cn}_\infty}}

\def\Sp{\mathtt{Sp}}

\def\hpfd{\mathtt{HPf_{dg}}}
\def\hpf{\mathtt{HPf_\infty}}

\def\CAlg{\mathtt{C^*}}
\def\catex{\mathtt{Cat^{ex}_\infty}}
\def\catinf{\mathtt{Cat_\infty}}
\def\hpfw{\mathtt{HPf}}

\def\KK{\mathrm{KK}}
\def\E{\mathrm{E}}

\def\H{\mathrm{H}}

\def\Sigmat{\Sigma^\infty_S}

\def\cC{\mathcal C}

\def\1{\bf{1}}

\def\Csep{\mathtt{SC^*}}
\def\G{\mathrm{G}}

\def\bu{\mathtt{bu}}
\def\iCsep{\text{$\mathtt{SC_\infty^*}$}}

\def\uloc{\mathcal{U}_{\mathtt{loc}}}
\def\mloc{\mathcal{M}_{\mathtt{loc}}}

\def\Map{\mathrm{Map}}
\def\bKn{\mathbf{K}^{\mathtt{nc}}}
\def\bKc{\mathbf{K}^{\mathtt{c}}}
\def\cD{\mathcal D}
\def\Alg{\mathtt{Alg}}

\def\TT{\mathbb{T}}
\def\kk{\mathrm{kk}}

\def\cA{\mathcal{A}}
\def\cB{\mathcal{B}}

\newcommand{\map}{\rightarrow}
\newcommand{\functor}{\rightarrow}

\def\h{\mathtt{h}}

\def\Ind{\mathrm{Ind}}
\def\hosc{\mathtt{HoSC^*}}

\def\Oinf{\mathcal{O}_\infty}
\def\Fin{\mathtt{Fin}}
\def\N{\mathrm{N}}

\def\one{\mathbf{1}}
\def\cQ{\mathcal{Q}}
\newcommand{\beq}{\begin{eqnarray}}
\newcommand{\beqn}{\begin{eqnarray*}}
\newcommand{\eeq}{\end{eqnarray}}
\newcommand{\eeqn}{\end{eqnarray*}}

\theoremstyle{definition}
\newtheorem{thm}{Theorem}[section]
\theoremstyle{definition}

\newtheorem{dis}[thm]{Disambiguation}

\newtheorem{thmquote}{Theorem}

\newtheorem{lem}[thm]{Lemma}
\newtheorem{prop}[thm]{Proposition}
\newtheorem{cor}[thm]{Corollary}
\newtheorem{ex}[thm]{Example}
\newtheorem{defn}[thm]{Definition}
\newtheorem{rem}[thm]{Remark}

\begin{document}

\title{Symmetric monoidal noncommutative spectra, strongly self-absorbing $C^*$-algebras, and bivariant homology}
\author{Snigdhayan Mahanta}
\email{snigdhayan.mahanta@mathematik.uni-regensburg.de}
\address{Fakult{\"a}t f{\"u}r Mathematik, Universit{\"a}t Regensburg, 93040 Regensburg, Germany.}
\subjclass[2010]{19Dxx, 46L85; 46L80, 55Nxx, 55P42}
\keywords{noncommutative spectra, stable $\infty$-categories, strongly self-absorbing $C^*$-algebras, connective $\E$-theory, noncommutative motives, $\TT$-duality}
\thanks{The third author was supported by the Deutsche Forschungsgemeinschaft (SFB 878 and SFB 1085), ERC through AdG 267079, and the Alexander von Humboldt Foundation (Humboldt Professorship of Michael Weiss).}

%%%%%%%%%%%%%%%%%%%%%%%%%%%%%%%%%%%%%%%%%%%%%%%%%%%%%%%%
\begin{abstract}
Continuing our project on noncommutative (stable) homotopy we construct symmetric monoidal $\infty$-categorical models for separable $C^*$-algebras $\iCsep$ and noncommutative spectra $\iNSp$ using the framework of Higher Algebra due to Lurie. We study smashing (co)localizations of $\iCsep$ and $\iNSp$ with respect to strongly self-absorbing $C^*$-algebras. We analyse the homotopy categories of the localizations of $\iCsep$ and give universal characterizations thereof. We construct a stable $\infty$-categorical model for bivariant connective $\E$-theory and compute the connective $\E$-theory groups of $\Oinf$-stable $C^*$-algebras. We also introduce and study the nonconnective version of Quillen's nonunital $\K'$-theory in the framework of stable $\infty$-categories. This is done in order to promote our earlier result relating topological $\TT$-duality to noncommutative motives to the $\infty$-categorical setup. Finally, we carry out some computations in the case of stable and $\Oinf$-stable $C^*$-algebras.
\end{abstract}

\maketitle

\setcounter{tocdepth}{2}
\tableofcontents

\begin{center}
{\bf Introduction}
\end{center}

The triangulated noncommutative stable homotopy category $\NSH$ (see Remark \ref{terminology}) was constructed in \cite{ThomThesis} as the receptacle of the universal triangulated homology theory for separable $C^*$-algebras. Its underlying additive category already appeared in the seminal paper of Connes--Higson on bivariant $\E$-theory \cite{ConHig} (see also \cite{HigETh}). In \cite{MyNSH} we constructed a stable presentable $\infty$-category of noncommutative spectra $\iNSp$ and used it to prove that $\NSH$ is a topological triangulated category as defined by Schwede \cite{SchTopTri}. The stable $\infty$-category $\iNSp$ is an ideal framework for stable homotopy theory of (pointed) noncommutative spaces.  Nevertheless, a very important part of the homotopy theory package, viz., the symmetric monoidal structure was left out of the discussion in \cite{MyNSH}. This is a glaring omission as it lies at the heart of our goal of developing a state of the art homotopy theoretic package for operator algebras with new features (see Remark \ref{THH}). One of the objectives of this article is to prove (using the formalism of Higher Algebra due to Lurie \cite{LurHigAlg})

\begin{thmquote}[Theorem \ref{SMNSp}]
There is a closed symmetric monoidal and compactly generated stable $\infty$-category of noncommutative spectra $\iNSp$. 
\end{thmquote}

The {\em strongly self-absorbing $C^*$-algebras} play a pivotal role in the Classification Program for $C^*$-algebras \cite{TomWin}. They are automatically simple and nuclear. Prominent examples of such $C^*$-algebras, that are also purely infinite, are Cuntz algebras $\mathcal{O}_2$, $\Oinf$, and tensor products of UHF algebras of infinite type with $\Oinf$. We are interested in strongly self-absorbing $C^*$-algebras because we construct {\em smashing localizations} of the $\infty$-category of separable $C^*$-algebras $\iCsep$ with respect to them. We describe the homotopy categories of the localized $\infty$-categories (see Proposition \ref{DStab}). At the level of homotopy categories we also characterize them by universal properties (see Theorem \ref{Univ}). The objects of $\iCsep^\op$ are compact in the $\infty$-category of pointed noncommutative spaces $\iNS$. Thus it is interesting to observe that nontrivial smashing (co)localizations exist already within the subcategory of $\iNS$ spanned by the (co)compact objects in contrast to the scenario of pointed spaces.

It was observed in \cite{MyNSH} that the homotopy category of noncommutative spectra $\h\iNSp$ is not an {\em algebraic} triangulated category and the question was raised whether it contains algebraic triangulated subcategories, which would facilitate computations enormously. With an eye towards such algebraization problems we colocalize the stable $\infty$-category $\iNSp$ with respect to the {\em suspension spectrum} (see Definition \ref{susSpec}) of any strongly self-absorbing $C^*$-algebra. Observe that colocalizations of $\iNSp$ can be viewed as $\infty$-subcategories spanned by certain colocal objects. The triangulated category of bivariant connective $\E$-theory for separable $C^*$-algebras, denoted by $\bu$, was constructed by Thom in \cite{ThomThesis} as a generalization of connective $\mathrm{kk}$-theory of (pointed compact metrizable) spaces \cite{SegKHom} (see also \cite{DadMcC}). In terms of computational complexity bivariant connective $\E$-theory lies in between noncommutative stable homotopy and bivariant (nonconnective) $\E$-theory. We construct a closed symmetric monoidal and compactly generated stable $\infty$-category $(\Ecn)^\op:=X'^{-1}\iNSp$ as an accessible localization of $\iNSp$ and show \begin{thmquote}[Theorem \ref{buTop}]
 There is a fully faithful exact functor $\bu\hookrightarrow\hEcn$ thereby showing that $\bu$ is a topological triangulated category.
\end{thmquote} 
\noindent
We completely describe the subcategory of $\hEcn$ spanned by the suspension spectra of $C^*$-algebras after (co)localization with respect to a purely infinite strongly self-absorbing $C^*$-algebra satisfying UCT (cf. Theorems \ref{OColoc}, \ref{UHF}, and \ref{O2}). Although these results do not entirely settle the algebraization problem, they demonstrate that certain (co)localized subcategories of $\iNSp$ are amenable to computation as they reduce to familiar bivariant homology theories. A consequence of Theorem \ref{OColoc} is that the canonical map from connective $\E$-theory to topological $\K$-theory is an isomorphism for $\Oinf$-stable $C^*$-algebras (see Remarks \ref{ECtoE} and \ref{OComp}). It turns out that the noncommutative stable cohomotopy of any $\Oinf$-stable $C^*$-algebra contains its topological $\K$-theory as a summand (see Proposition \ref{Ksummand} for a more general result). Using these results we compute the connective $\E$-theory and identify a summand in the noncommutative stable cohomotopy groups of $ax+b$-semigroup $C^*$-algebras associated with number rings (see Theorem \ref{AX+B} and Remark \ref{ax+bsum}). The colocalization with respect to $\mathcal{Z}$, where $\mathcal{Z}$ is the Jiang--Su algebra, is the most interesting case from the viewpoint of the classification program and for some partial results in this direction see Section 4 of \cite{MyColoc}.

Algebraic $\K$-theory does not (directly) make sense for topological spaces. The appropriate theory in this context is Waldhausen's $\mathrm{A}$-theory \cite{Waldhausen}, which is homotopy invariant but not excisive. One needs the theory of functor calculus to analyse it. However, algebraic $\K$-theory does make sense for certain noncommutative spaces. One can view a compact noncommutative space or a unital $C^*$-algebra simply as a unital complex algebra and study its algebraic $\K$-theory. Although algebraic $\K$-theory satisfies excision on the category of $C^*$-algebras \cite{SusWod2}, it is not homotopy invariant. Roughly speaking, a spectrum valued functor $F$ on $k$-algebras satisfies excision, where $k$ is a field, if for every short exact sequence $0\map A\map B\map C\map 0$ the induced diagram $F(A)\map F(B)\map F(C)$ is a homotopy (co)fiber sequence. Thus one needs algebraic $\K$-theory to treat unital and nonunital algebras on an equal footing (note that $A$ is strictly nonunital unless the extension is trivial). Quillen introduced a $\K'_0$-theory for nonunital algebras in \cite{QuiNonunitalK0}, whose higher (connective) version was developed by the author in \cite{KQ}. The author's motivation in that article was the categorification of topological $\TT$-duality. The higher version of $\K'_0$-theory was called $\KQ$-theory by the author in \cite{KQ} so that a conflict with $\G$-theory (or $\K'$-theory of pseudo-coherent modules) could be avoided. In the final part of this article we define nonconnective $\KQ$-theory and show that for stable and $\Oinf$-stable $C^*$-algebras it agrees naturally with their nonconnective algebraic as well as topological $\K$-theory (see Theorem \ref{KQ} for a more general result). These results reinforce the idea that certain (co)homological invariants tend to become more tractable after specific (co)localizations. From the computational viewpoint the following picture emerges: \begin{thmquote}[Remark \ref{identical}]
 For stable and $\Oinf$-stable separable $C^*$-algebras the four possible invariants, viz., connective $\E$-theory, nonconnective $\KQ$-theory, nonconnective algebraic $\K$-theory, and topological $\K$-theory are all naturally isomorphic. 
\end{thmquote} At least the assertion in the $\Oinf$-stable case for all four invariants appears to be new (see also \cite{CorPhi,MyComparison}). The results in this part rely on various properties of algebraic $\K$-theory in the setting of stable $\infty$-categories established by Blumberg--Gepner--Tabuada \cite{BluGepTab}. Given any separable $C^*$-algebra $A$ we associate with it its ($\cpt$-stabilized) noncommutative motive $\cM_\infty^\cpt(A)$ that takes values in the stable presentable $\infty$-category of noncommutative motives $\mloc$ constructed in \cite{BluGepTab}. We then generalize our earlier result on categorification of topological $\TT$-duality \cite{KQ} to pointed noncommutative spaces in the setting of (stable) $\infty$-categories as follows (see also Remark \ref{categorification}): \begin{thmquote}[Theorem \ref{facKK}]
 The functor $(\cM_\infty^\cpt)^\op:\iCsep^\op\functor\mloc^\op$ induces the following two functors: \begin{enumerate}
                                                                                \item $\iNS\functor\mloc^\op$ that is continuous, and
                                                                                \item $\ikk\functor\mloc$ that is exact.
                                                                               \end{enumerate}
\end{thmquote} 

\noindent
Here $\ikk$ is a stable $\infty$-categorical incarnation of bivariant $\K$-theory that we constructed in \cite{MyColoc}. The bivariant $\E$-theory counterpart $\mathtt{E_\infty}$ of $\ikk$ can be constructed similarly (see Remark \ref{Einfty}) and one may replace $\ikk$ by $\mathtt{E_\infty}$ in the above result. Our article also demonstrates that noncommutative motives constitute a bivariant homology theory on the category of separable $C^*$-algebras under favourable circumstances (see Remark \ref{BivHomColoc}). 

\begin{rem} \label{THH}
The above result opens up the prospect of studying topological Hochschild (resp. cyclic) homomology of noncommutative (ring) spectra via noncommutative motives. In fact, one of our motivations behind the introduction of the symmetric monoidal structures on $\iNS$ and $\iNSp$ is the direct study of such theories that will be explored elsewhere. 
\end{rem} In this article we have decided to change some terminology (see Disambiguation \ref{warning}) used previously by the author in order to align ourselves with the conventions in topology. For the benefit of the reader we record them here:

\begin{itemize}
 \item $\NSH^\op =$ noncommutative stable homotopy category,
 \item $(\NSH^f)^\op =$ homotopy category of noncommutative finite spectra,
 \item $\NSH(\CC,A) =$ noncommutative stable cohomotopy of $A$,
 \item $\NSH(A,\CC) =$ noncommutative stable homotopy of $A$.
\end{itemize}

\medskip
\noindent
{\bf Notations and conventions:} Throughout this article $\prot$ will denote the maximal $C^*$-tensor product. All $C^*$-algebras are assumed to be separable unless otherwise stated. For any $\infty$-category $\cC$ we denote by $\h\cC$ its homotopy category. In the context of $\infty$-categories a functor (resp. limit or colimit) will implicitly mean an $\infty$-functor (resp. $\infty$-limit or $\infty$-colimit). There is a Yoneda embedding $j:\iCsep^\op\map\iNS$ and a separable $C^*$-algebra $A$ is viewed as a noncommutative space via $j(A)$. In the sequel for brevity we suppress $j$ from the notation. 

The stable $\infty$-category of noncommutative spectra $\iNSp$ is a localization (with respect to a set of maps $S$) of the stabilization of pointed noncommutative spaces $\iNS$. This localization is performed in order to achieve optimal excision property following \cite{ThomThesis}. However, there are other conceivable choices for the set of maps to localize (see the set of maps $S'\subset S$ in Remark 2.29 of \cite{MyNSH}). Strictly speaking, one should denote the category $\iNSp$ by $\iNSp[T^{-1}]$ (or something similar) to indicate the dependence on the localization with respect to a set of maps $T$. For brevity we have chosen the concise notation $\iNSp$ with the understanding that there is an implicit localization with respect to $T=S$ (other possibilities being $T= S',\emptyset$).

\medskip
\noindent
{\bf Acknowledgements:} The author would like to thank T. Nikolaus and W. Winter for helpful discussions. The author is also grateful to S. Barlak for his feedback. The author has benefited from the hospitality of Max Planck Institute for Mathematics, Bonn under various stages of development of this project.

\section{The symmetric monoidal version of noncommutative spectra}
Recall from \cite{MyNSH} that there is an $\infty$-category of pointed noncommutative spaces $\iNS$ as well as a stable $\infty$-category of noncommutative spectra $\iNSp$, which is obtained after a localization of the stabilization of the $\infty$-category $\iNS$. In this section we construct a closed symmetric monoidal structure on $\iNS$ (resp. $\iNSp$) generalizing the smash product of pointed finite CW complexes (resp. finite spectra). 

Let $\Fin_*$ denote the category, whose objects are pointed sets $\langle n\rangle =\{*,1,\cdots, n\}$ with $*$ being the basepoint and whose morphisms are pointed maps. Let $\N(\Fin_*)$ denote its nerve. A {\em symmetric monoidal $\infty$-category} $\cC^\otimes$ is a coCartesian fibration of simplicial sets $p:\cC^\otimes\map\N(\Fin_*)$ with the property: for each $n\geqslant 0$ there is an equivalence $\cC^\otimes_{\langle n\rangle}\simeq (\cC^\otimes_{\langle 1\rangle})^n$ induced by the maps $\{\rho^i:\langle n\rangle \map \langle 1\rangle\}_{1\leqslant i\leqslant n}$. One should regard $\cC:=\cC^\otimes_{\langle 1\rangle}$ as the $\infty$-category, which is symmetric monoidal. It is customary to work with the underlying symmetric monoidal category $\cC$, leaving out the rest of the structure as implicitly understood. A symmetric monoidal $\infty$-category can also be regarded as a commutative monoid object in $\catinf$, which is the $(\infty,1)$-category of (small) $\infty$-categories (see Definition 3.0.0.1 of \cite{LurToposBook}). Recall that a symmetric monoidal presentable $\infty$-category is said to be {\em closed} if the tensor product preserves colimits separately in each variable. For further details the readers may consult Chapters 2, 3, and 4 of \cite{LurHigAlg}.

\begin{prop} \label{closedSM}
 The $\infty$-categories $\iCsep$ and $\iNS := \Ind(\iCsep^\op)$ are symmetric monoidal. Moreover, the tensor product $\otimes:\iNS\times\iNS\functor\iNS$ preserves small colimits in each variable separately, i.e., $\iNS$ is closed symmetric monoidal and the Yoneda functor $j:\iCsep^\op\functor\iNS$ is symmetric monoidal.
\end{prop}

\begin{proof}
 It is well-known that the topological category $\Csep$ is symmetric monoidal under the maximal $C^*$-tensor product $\prot$. One can verify that $\prot:\Csep\times\Csep\map\Csep$ is a continuous functor that is compatible with the associativity and unit constraints; moreover, the symmetry maps $\eta_{A,B}:A\prot B\map B\prot A$ for all $A,B\in\Csep$ constitute a natural transformation of continuous functors. Hence its topological nerve $\iCsep$ is a symmetric monoidal $\infty$-category. The symmetric monoidal structure on $\iCsep$ endows $\iCsep^\op$ with a symmetric monoidal structure $\otimes$ that is uniquely defined up to a contractible space of choices (see Remark 2.4.2.7 of \cite{LurHigAlg}). One needs to verify that $\otimes$ commutes with finite colimits in $\iCsep^\op$. For this note that the maximal $C^*$-tensor product $\prot$ preserves homotopy pullbacks in the topological category $\Csep$ (see Corollary 1.9, Proposition 1.11 of \cite{SchTopMet3} and Remark 3.10 of \cite{PedCW}). Now all other assertions follow from Corollary 4.8.1.13 of \cite{LurHigAlg}. 
\end{proof}

\noindent
Note that the $\infty$-category $\iNS$ is pointed and it follows from Proposition 4.8.2.11 of \cite{LurHigAlg} that there is an equivalence $\iNS\simeq\iNS\otimes\S_*:=\Fun^\mathrm{R}(\iNS^\op,\Sp)$. Here $\otimes$ is the one in the symmetric monoidal $\infty$-category of presentable $\infty$-categories with colimit preserving functors.

\begin{lem} \label{coSM}
 The stabilization $\Sp(\iNS)$ of $\iNS$ is a closed symmetric monoidal stable presentable $\infty$-category and the functor $\Sigma^\infty:\iNS\functor\Sp(\iNS)$ is symmetric monoidal.
\end{lem}

\begin{proof}
The assertion is a consequence of the Proposition \ref{closedSM} and Proposition 4.8.2.18 of \cite{LurHigAlg} (see also Theorem 5.1 of \cite{GepGroNik}). A more step-by-step approach is to argue via the identification of stable $\infty$-categories $\Sp(\iNS)\simeq\iNS\otimes\Sp$ (see Example 4.8.1.22 of \cite{LurHigAlg}). Using the stabilization $\Sigma^\infty:\S_*\functor\Sp$ of pointed spaces, the stabilization of pointed noncommutative spaces can be regarded as the composite $$\iNS\simeq\iNS\otimes\S_*\overset{}{\map}\iNS\otimes\Sp\simeq\Sp(\iNS).$$\end{proof}

\noindent
For any $*$-homomorphism $f:B\map C$ in $\Csep$ there is a canonical map $\theta(f):\ker(f)\map\C(f)$ in $\iCsep$, where $\C(f)$ denotes the mapping cone of $f$. The map $\theta(f)$ can also be viewed as an element in $\iCsep^\op(\C(f),\ker(f))$. Let $$T_0=\{\C(f)\map\ker(f) \,|\, f: A\map B \text{ surjective in $\Csep$} \}$$ denote a small set of morphisms in $\iCsep^\op$. Recall from \cite{MyNSH} that there is a functor $\St:\iCsep^\op\functor\Sp(\iNS)$ that arises as a composition of two symmetric monoidal functors $$\iCsep^\op\overset{j}{\functor}\iNS\overset{\Sigma^\infty}{\functor}\Sp(\iNS),$$ such that its image lies inside the compact objects of $\Sp(\iNS)$. This functor descends to a symmetric monoidal functor between tensor triangulated categories $\St:\hosc[\Sigma^{-1}]\functor\h\Sp(\iNS)$, where $\hosc[\Sigma^{-1}]$ is the Spanier--Whitehead category of the homotopy category of $\Csep$ with respect to $\Sigma(-)$. We construct a {\em strongly saturated collection} of morphisms $S$ (see Definition 5.5.4.5 of \cite{LurToposBook}) in $\Sp(\iNS)$, which is of small generation starting from the small set $$S_0=\{\St(\theta)\,|\, \theta\in T_0\}$$ that is compatible with the triangulation as follows: Let $\cA$ denote the stable $\infty$-subcategory of $\Sp(\iNS)$ generated by the set $\{\mathrm{cone}(\phi)\,|\, \phi\in S_0\}$. Then $\Ind_\omega(\cA)$ is a stable presentable $\infty$-subcategory of $\Sp(\iNS)$ (see Proposition 1.1.3.6 of \cite{LurHigAlg}). Let $S$ denote the class of maps in $\Sp(\iNS)$, whose cones lie in the essential image of $\Ind_\omega(\cA)$. We deduce from Proposition 5.6 of \cite{BluGepTab} that $S$ is a strongly saturated collection of small generation. The cofiber of the map $\Ind_\omega(\cA)\map\Sp(\iNS)$ can be identified with the accessible localization $L_S:\Sp(\iNS)\functor S^{-1}\Sp(\iNS)$ (see Proposition 5.6 of \cite{BluGepTab}). Note that in \cite{MyNSH} we defined the stable $\infty$-category of {\em noncommutative spectra} as $\iNSp= S^{-1}\Sp(\iNS)$ and denoted the localization functor $L_S:\Sp(\iNS)\functor S^{-1}\Sp(\iNS)$ simply by $L$.

\begin{thm} \label{SMNSp}
 There is a colimit preserving symmetric monoidal functor $\Sigmat= L_S\circ\Sigma^\infty:\iNS\functor\iNSp$ between presentable closed symmetric monoidal $\infty$-categories.
\end{thm}

\begin{proof}
We only need to show that the localization $L_S$ is symmetric monoidal, i.e., whenever $X\map Y$ is a $L_S$-equivalence, then so is $X\otimes Z\map Y\otimes Z$ for every $Z\in\Sp(\iNS)$ (see Proposition 2.2.1.9 and Example 2.2.1.7 of \cite{LurHigAlg}). By construction $\iNS$ is compactly generated by the objects of $\iCsep^\op$. Consequently $\St(\iCsep^\op)$ lands inside the $\infty$-subcategory of compact objects of $\Sp(\iNS)$. It follows from Proposition 5.7 of \cite{BluGepTab} that there is an equivalence of stable $\infty$-categories $S_0^{-1}\Sp(\iNS)\simeq S^{-1}\Sp(\iNS)=\iNSp$. The tensor product $\otimes$ on $\Sp(\iNS)$ preserves colimits separately in each variable (see, for instance, Theorem 5.1 of \cite{GepGroNik}) and $\Sp(\iNS)$ is compactly generated by $\St(\hosc[\Sigma^{-1}]^\op)$. By Theorem 2.26 of \cite{MyNSH} it suffices to show: if $\theta(f):\ker(f)\map\C(f)$ is the canonical map in $\iCsep$ for any surjection $f:A\map B$ in $\Csep$, then for any $C\in\Csep$ the map $\theta(f)\otimes \id_C: \ker(f)\prot C\map\C(f)\prot C$ is the same as $\theta(f\otimes\id_C):\ker(f\otimes\id_C)\map\C(f\otimes\id_C)$. This can be verified using the exactness of the maximal $C^*$-tensor product (see Lemma 4.1 of \cite{GueHigTro}).
\end{proof}

\begin{defn} \label{susSpec}
 We call $\Sigmat:\iNS\functor\iNSp$ the {\em suspension spectrum} of pointed noncommutative spaces. The composite functor $\Sigmat\circ j:\iCsep^\op\functor\iNSp$ was denoted by $\st$ in \cite{MyNSH}.
\end{defn}

\begin{rem}\label{terminology}
Since $\h\iCsep^\op$ is the homotopy category of pointed noncommutative (compact metrizable) spaces, it seems very natural to consider $\NSH^\op$ as its suspension stabilization. Thus we propose to (re)define $$\NSH^\op = \text{ noncommutative stable homotopy category,}$$ deviating from the terminology in \cite{ThomThesis,MyNSH,MyNGH}. Naturally we refer to its triangulated subcategory $(\NSH^f)^\op$ as the homotopy category of noncommutative finite spectra; see Definition 2.1 of \cite{MyNGH}, where $\NSH^f$ was called the homotopy category of noncommutative finite spectra.
\end{rem}

\begin{rem}
The category $\NSH$ is a tensor triangulated category with respect to $\prot$ (see \cite{ThomThesis,Ambrogio}. The homotopy category of noncommutative spectra $\hNSp$ is also a tensor triangulated category, containing $\NSH^\op$ as a full tensor triangulated subcategory via the functor $\pi^\op:\NSH^\op\functor\h\iNSp$ (see Theorem 2.26 of \cite{MyNSH}). It also contains $(\NSH^f)^\op$, defined in \cite{MyNGH}, as a full tensor triangulated subcategory.
\end{rem}

\begin{rem}
 Inside $\iNS$ it is possible to identify an $\infty$-subcategory where {\em surjections in $\Csep$ behave like cofibrations}. Let $T$ be the strongly saturated collection generated by $j(T_0)$ inside $\iNS$. We may construct an accessible localization $L_T:\iNS\functor T^{-1}\iNS$ with respect to $T$. The presentable $\infty$-category $T^{-1}\iNS$ has the desired property and the suspension spectrum functor $\Sigmat:\iNS\functor\iNSp$ factors through $T^{-1}\iNS$.
\end{rem}

\section{Strongly self-absorbing $C^*$-algebras and localizations of $\iCsep$}
A separable unital $C^*$-algebra $\cD$ ($\cD\neq\CC$) is called {\em strongly self-absorbing} if the there is an isomorphism $\phi:\cD\map\cD\prot\cD$ that is approximately unitarily equivalent to $\id_\cD\otimes \one_\cD$ \cite{TomWin}. In \cite{TomWin} the authors introduced and conducted an elaborate study of strongly self-absorbing $C^*$-algebras mainly with applications to the Elliott's Classification Program in mind. These $C^*$-algebras share certain properties similiar to $\cpt$, i.e., the $C^*$-algebra of compact operators on a separable Hilbert space. We are going to use these $C^*$-algebras to construct interesting (co)localizations of noncommutative spaces and spectra.

\begin{rem} \label{SSA}
 In \cite{DadWin} the authors showed the for any strongly self-absorbing $C^*$-algebra $\cD$ the map $\id_\cD\otimes\one_\cD$ is homotopic to an isomorphism $\phi:\cD\map\cD\prot\cD$. In \cite{DadWin} the result was asserted under the $\K_1$-injectivity condition, which later turned out to be redundant (see Remark 3.3. of \cite{WinZStable}).
\end{rem}

Let $\cC$ be a symmetric monoidal $\infty$-category with unit object $\one$. Then a map $e:\one\map E$ exhibits $E$ as an {\em idempotent object} if $\id_E\otimes e: E\simeq E\otimes\one\map E\otimes E$ is an equivalence in $\cC$ (see, for instance, Definition 4.8.2.1 of \cite{LurHigAlg}). Unlike finite CW complexes it is possible to find several interesting idempotent objects in the world of separable $C^*$-algebras, which are compact objects in $\iNS$. Indeed, we find

\begin{lem} \label{SSAidem}
 Any strongly self-absorbing $C^*$-algebra $\cD$ is an idempotent object in $\iCsep$. The same assertion holds for $\cpt$.
\end{lem}

\begin{proof}
 For a strongly self-absorbing $C^*$-algebra $\cD$ the canonical unital $*$-homomorphism $\CC\map\cD$ exhibits it as an idempotent object in $\iCsep$ (see Remark \ref{SSA}). For $\cpt$ the map $\CC\map\cpt$ sending $1\mapsto e_{11}$ exhibits $\cpt$ as an idempotent object in $\iCsep$.
\end{proof}

\begin{rem}
If $E\in\cC$ is an idempotent object, then $L_E:\cC\map\cC$ of the form $L_E(X)=-\otimes E$ is a localization. In \cite{GepGroNik} the authors called localizations $L_E:\cC\map\cC$ of the form $L_E(X)=-\otimes E$ for some $E\in\cC$ {\em smashing localizations} in keeping with the terminology prevalent in stable homotopy theory. Any smashing localization $L_E:\cC\map\cC$ is compatible with the symmetric monoidal structure on $\cC$ and, in fact, $L_E\cC$ inherits a symmetric monoidal structure from $\cC$, such that $L_E:\cC\map L_E\cC$ becomes symmetric monoidal (see Proposition 2.2.1.9 and Proposition 4.8.2.7 of \cite{LurHigAlg}). By abuse of notation we are sometimes going to drop the object $E$ from the smashing localization $L_E$ and denote it simply by $L$.
\end{rem}

\begin{ex}
Smashing localizations of the $\infty$-category of separable $C^*$-algebras $\iCsep$ produces interesting results. By definition $\iCsep$ is opposite to the $\infty$-category of pointed noncommutative compact Hausdorff spaces. We present a few pertinent examples here.

\begin{enumerate} \label{ConE}
\item If $L(A)=A\otimes\cpt$, then we denote the smashing localization $L\iCsep$ by $\iCsep[\cpt^{-1}]$. It is the $\infty$-category of {\em $C^*$-stable $C^*$-algebras}. For finite pointed CW complexes $(X,x)$ and $(Y,y)$ the homotopy set $\h\iCsep[\cpt^{-1}](L(\C(X,x)),L(\C(Y,y)))$ is the bivariant connective $\E$-theory group denoted by $\kk((Y,y),(X,x))$ in \cite{DadMcC} (see Remark \ref{HtpyIsom} below).

\item If $L(A)=A\otimes \cD$, where $\cD$ is a strongly self-absorbing $C^*$-algebra, then we denote the smashing localization $L \iCsep$ by $\iCsep[\cD^{-1}]$. We refer to it as the $\infty$-category of {\em $\cD$-stable $C^*$-algebras}. From the perspective of Elliott's Classification Program the $\infty$-category $\iCsep[\mathcal{Z}^{-1}]$ would be the most interesting localization, where $\mathcal{Z}$ is the Jiang--Su algebra. We call it the $\infty$-category of {\em $\mathcal{Z}$-stable $C^*$-algebras}.

\item Let $\Oinf$ be the universal unital Cuntz algebra on generators $\{s_i,s_i^*\,|\,i\in\NN\}$ satisfying $s_i^*s_j=\delta_{ij}$. If $\cD=\Oinf$ we call $\iCsep[\Oinf^{-1}]$ the $\infty$-category of {\em $\Oinf$-stable $C^*$-algebras}. The suspension stable version of this category will be analysed in the next section.
\end{enumerate}
\end{ex}
 
 \begin{prop}
  Let us suppose that there is a unital embedding $\iota_\cD:\cD\map\cD'$ of strongly self-absorbing $C^*$-algebras. Then $\cD'$ is an idempotent object in $\iCsep[\cD^{-1}]$.
 \end{prop}

 \begin{proof}
 Consider the following commutative diagram in $\Csep$
  
  \beqn
  \xymatrix{
  \cD'\ar[rr]^{\id_{\cD'}\otimes \one_{\cD'}}\ar[rd]_{\id_{\cD'}\otimes\one_\cD} && \cD'\prot\cD'\\
  &\cD'\prot\cD\ar[ru]_{\id_{\cD'}\otimes\iota_\cD}.
  }
  \eeqn Since $\cD'$ is strongly self-absorbing $\id_{\cD'}\otimes\one_{\cD'}$ is homotopic to an isomorphism $\cD'\map\cD'\prot\cD'$. It follows from Proposition 5.12 of \cite{TomWin} that $\id_{\cD'}\otimes\one_\cD$ is homotopic to an isomorphism $\cD'\map\cD'\prot\cD$ demonstrating that $\cD'$ is $\cD$-stable whence $\cD'\in\iCsep[\cD^{-1}]$. It follows that $\id_{\cD'}\otimes\iota_\cD$ is a homotopy equivalence. Observe that the unit object in $\iCsep[\cD^{-1}]$ is $\cD$. Thus the unital embedding $\iota_\cD:\cD\map\cD'$ exhibits $\cD'$  as an idempotent object in $\iCsep[\cD^{-1}]$. 
 \end{proof}

 \begin{cor}
  In the localized $\infty$-category $\iCsep[\mathcal{Z}^{-1}]$ every strongly self-absorbing $C^*$-algebra is an idempotent object. 
 \end{cor}
 
 \begin{proof}
  The assertion follows from the characterization of $\mathcal{Z}$ as the initial object in the homotopy category of strongly self-absorbing $C^*$-algebras with unital $*$-homomorphisms (see Corollary 3.2 of \cite{WinZStable}).
 \end{proof}

 \begin{rem}
 In view of the above Corollary one may construct $\iCsep[\cD^{-1}]$ for any strongly self-absorbing $C^*$-algebra $\cD$ as a localization of $\iCsep[\mathcal{Z}^{-1}]$. Thus equivalences in $\iCsep[\mathcal{Z}^{-1}]$ contain the most refined information amongst all smashing localizations with respect to strongly self-absorbing $C^*$-algebras.
 \end{rem}

  \noindent
  For any $A,B\in\Csep$ we denote by $[A,B]$ the homotopy classes of $*$-homomorphisms $A\map B$.
  
\begin{prop} \label{DStab}
For any $A,B\in\Csep$ and any strongly self-absorbing $C^*$-algebra $\cD$ there is a natural isomorphism $$\h\iCsep[\cD^{-1}](L(A),L(B))\cong [A,B\prot\cD].$$ 
\end{prop}

\begin{proof}
Let us first observe that there is an identification $$\h\iCsep[\cD^{-1}](L(A),L(B))\cong\h\iCsep(A\prot\cD,B\prot\cD).$$ There is an element $\theta_A=\id_A\otimes\one_\cD\in\Csep(A,A\prot\cD)$ sending $a\mapsto a\otimes \one_\cD$. This induces a map $$K:\h\iCsep(A\prot\cD,B\prot\cD)\map\h\iCsep(A,B\prot\cD)$$ by precomposing with $[\theta_A]$ (here $[-]$ denotes the homotopy class). Using the fact that $\id_{\cD}\otimes \one_{\cD}:\cD\map\cD\prot\cD$ is homotopic to an isomorphism $\gamma:\Csep(\cD,\cD\prot\cD)$, we deduce that the map $\id_B\otimes\id_{\cD}\otimes \one_{\cD}$ is homotopic to an isomorphism $\gamma_B\in\Csep(B\prot\cD,B\prot\cD\prot\cD)$. Now we define a map $$M:\h\iCsep(A,B\prot\cD)\map\h\iCsep(A\prot\cD,B\prot\cD)$$ as follows: $M([\phi])=[\gamma_B^{-1}\circ (\phi\otimes\id_{\cD})]$. Observe that $K\circ M([\phi]) = [\gamma_B^{-1}\circ (\phi\otimes\id_{\cD})]\circ [\theta_A]=[\gamma_B^{-1}\circ (\id_B\otimes\id_\cD\otimes\one_\cD)\circ\phi]$. Since $[\id_B\otimes\id_\cD\otimes\one_\cD]=[\gamma_B]$ the composition $K\circ M = \id : \h\iCsep(A,B\prot\cD)\map\h\iCsep(A,B\prot\cD)$.

Now $M\circ K([\psi])=M([\psi\circ\theta_A])=[\gamma_B^{-1}\circ ((\psi\circ\theta_A)\otimes\id_\cD)]$. Let $\tau_\cD:\cD\map\cD$ denote the tensor flip map, which is also homotopic to the identity. A verification on the simple tensors demonstrates that $[(\id_B\otimes\tau_\cD) \circ ((\psi\circ\theta_A)\otimes\id_\cD)]= [\gamma_B \circ\psi]$. It follows that $M\circ K = \id: \h\iCsep(A\prot\cD,B\prot\cD)\map\h\iCsep(A\prot\cD,B\prot\cD)$. It remains to observe that $\h\iCsep(A,B\prot\cD)\cong[A,B\prot\cD]$ (see Section 2.1 of \cite{MyNSH}).
\end{proof}

 \begin{prop} \label{OIdem}
 The $*$-homomorphism $\id\otimes\one_\cD: \cpt\map\cpt\prot\cD$ (resp. $\id_\cD\otimes e_{11}:\cD\map\cD\prot\cpt$) exhibits $\cpt\prot\cD$ (resp. $\cD\prot\cpt$) as an idempotent object in $\iCsep[\cpt^{-1}]$ (resp. $\iCsep[\cD^{-1}]$).
 \end{prop}

 \begin{proof}
  We only show that $\cpt\prot\cD$ is an idempotent object in $\iCsep[\cpt^{-1}]$. The proof of the other assertion is similar. Consider the diagram $\CC\overset{\iota}{\map}\cpt\overset{\id\otimes\one_\cD}{\map}\cpt\prot\cD$, where $\iota(1)= e_{11}$. Tensoring with $\cpt\prot\cD$ we get a diagram $\cpt\prot\cD\map\cpt\prot\cpt\prot\cD\map\cpt\prot\cD\prot\cpt\prot\cD$ whose composition is the $*$-homomorphism $(a\otimes x) \mapsto (e_{11}\otimes\one_\cD)\otimes (a\otimes x)$. According to the proof of Theorem 2.5 of \cite{DadPen} this composition is an equivalence. Moreover, the map $\cpt\prot\cD\map\cpt\prot\cpt\prot\cD$ is an equivalence in $\iCsep[\cpt^{-1}]$. Consequently, the map $\cpt\prot\cpt\prot\cD\map\cpt\prot\cD\prot\cpt\prot\cD$ is also an equivalence exhibitting $\cpt\prot\cD$ as an idempotent object in $\iCsep[\cpt^{-1}]$. 
 \end{proof}

\begin{rem}
 Let $\cD$ be any strongly self-absorbing $C^*$-algebra. It follows from Theorem 2.5 of \cite{DadPen} that the $*$-homomorphism $\CC\map\cD\prot\cpt$ sending $1\mapsto \one_\cD\otimes e_{11}$ exhibits $\cD\prot\cpt$ as an idempotent object in $\iCsep$. Moreover, the argument in the above Proposition \ref{DStab} goes through to show that $$\h\iCsep[(\cD\prot\cpt)^{-1}](L(A),L(B))\cong [A,B\prot(\cD\prot\cpt)].$$
\end{rem}
 
 \begin{cor} \label{OtoK}
  The $(\Oinf\prot\cpt)$-stable $\infty$-category $\iCsep[(\Oinf\prot\cpt)^{-1}]$ is equivalent to the localization of $\iCsep[\cpt^{-1}]$ with respect to $\Oinf\prot\cpt$.
 \end{cor}

 \begin{rem} \label{HtpyIsom}
 Consider the following problem: Given two connected finite pointed CW complexes $(X,x)$ and $(Y,y)$ are the $C^*$-algebras $\C(X,x)\prot\cpt$ and $\C(Y,y)\prot\cpt$ homotopy equivalent? The answer to the question can be detected in terms of a bivariant homology theory, viz., connective $\mathrm{kk}$-theory (see Theorem 2.4 of \cite{DadMcC}). The connective $\mathrm{kk}$-category for connected finite pointed CW complexes can be viewed within the localization $\iCsep[\cpt^{-1}]$ (cf. Example \ref{ConE} (1)) and it should not be confused with Cuntz $kk$-theory for $m$-algebras (or locally convex algebras). Homotopy equivalences of matrix bundles can also be detected by bivariant connective $\E$-theory \cite{ThomThesis}. In order to determine actual isomorphism types (not merely homotopy types) one needs sharper invariants \cite{DadPen}.
 \end{rem}

 \noindent
 Now we demonstrate that the homotopy category of the smashing localization $\h\iCsep[\cD^{-1}]$ admits a universal characterization much like $\KK$-theory. The localization functor $L_\cD:\iCsep\functor\iCsep[\cD^{-1}]$ induces a canonical (ordinary) functor $L_\cD:\Csep\functor\h\iCsep[\cD^{-1}]$. Recall that a functor $F:\Csep\functor\cC$ ($\cC$ an ordinary category) is called {\em $\cD$-stable} if $F$ sends the morphism $A\map A\prot\cD$ mapping $a\mapsto a\otimes \one_\cD$ to an isomorphism in $\cC$ for all $A\in\Csep$.
 
\begin{thm} \label{Univ}
 The functor $L_\cD:\Csep\functor\h\iCsep[\cD^{-1}]$ is the universal homotopy invariant and $\cD$-stable functor on $\Csep$.
\end{thm}

\begin{proof}
 Let us first show that functor $L_\cD$ is homotopy invariant and $\cD$-stable. It is easy to verify that it is homotopy invariant. It follows from the arguments in the proof of Proposition \ref{DStab} that the map $\h\iCsep[\cD^{-1}](L_\cD (A\prot\cD),L_\cD(B))\map\h\iCsep[\cD^{-1}](L_\cD (A),L_\cD(B))$ induced by $A\map A\prot\cD$ is an isomorphism for all $B\in\Csep$. For any $B\in\Csep$ the map $$\h\iCsep[\cD^{-1}](L_\cD(B),L_\cD (A))\map\h\iCsep[\cD^{-1}](L_\cD(B), L_\cD (A\prot\cD))$$ is equivalent to that map $[B,A\prot\cD]\map[B,A\prot\cD\prot\cD]$ once again by Proposition \ref{DStab}. This map is induced by $A\prot\cD\map A\prot\cD\prot\cD$ sending $a\otimes d \mapsto a\otimes\one_\cD\otimes d$. Since $\cD$ is strongly self-absorbing one easily sees $[B,A\prot\cD]\map[B,A\prot\cD\prot\cD]$ is an isomorphism. Since $L_\cD$ is surjective on objects we conclude that $L_\cD$ is $\cD$-stable.
 
 Let $F_i:\h\iCsep[\cD^{-1}]\functor\cC$ with $i=1,2$ be two functors making the following diagram commute
 
 \beq \label{exist}
 \xymatrix{
 \Csep\ar[rr]^{L_\cD}\ar[rd]_{F} && \h\iCsep[\cD^{-1}]\ar@{-->}[ld]^{F_i}\\
 & \cC.
 }
 \eeq On objects they are both determined by $\cD$-stability $F_i(A\prot\cD)\cong F(A\prot\cD)\cong F(A)$. Similarly, on each morphism $\phi: A\prot\cD\map B\prot\cD$ the value of $F_i(\phi)$ is uniquely determined by the following diagram:
 
 \beqn
 \xymatrix{
F_i(A\prot\cD) \ar[r]^{F_i(\phi)} & F_i (B\prot\cD)\\
F(A)\ar[r]^{F(\phi)}\ar[u]^{\cong} & F(B) \ar[u]_{\cong}.
}
 \eeqn
 
 For the existence note that for any homotopy invariant and $\cD$-stable functor $F:\Csep\functor\cC$ there is a functor $\overline{F}:\h\iCsep[\cD^{-1}]\functor\cC$ sending $A\prot\cD$ to $F(A\prot\cD)\cong F(A)$ that makes the above diagram \eqref{exist} commute (up to a natural isomorphism).
\end{proof}

\section{Bivariant connective $\E$-theory and (co)localizations of $\iNSp$}
Let us remind the readers that the functor $\st:\iCsep^\op\functor\iNSp$ arises as a composition of the following functors $$\iCsep^\op\overset{j}{\functor}\iNS\overset{\Sigma^\infty}{\functor} \Sp(\iNS)\overset{L_S}{\functor}S^{-1}\Sp(\iNS)=\iNSp.$$ For any separable $C^*$-algebra $A$ one ought to regard $\st(A)$ as its suspension spectrum after localization with respect to $S$. In the sequel we suppress the functor $j$ from the notation and simply write $\Sigmat (A)$ in place of $\st(A)=\Sigmat\circ j(A)$ for any separable $C^*$-algebra $A$.

\begin{defn}
Let $\cC$ be a symmetric monoidal $\infty$-category with unit object $\one$. We say that a map $e:E\map\one$ exhibits $E$ as a {\em coidempotent object} in $\cC$ if the dual map $e^\op:\one\map E$ exhibits $E$ as an idempotent object in $\cC^\op$. 
\end{defn} Recall that the symmetric monoidal structure on $\cC$ endows $\cC^\op$ with a symmetric monoidal structure that is uniquely defined up to a contractible space of choices.

\begin{lem} \label{SSAcoidem}
 If $\cD$ is a strongly self-absorbing $C^*$-algebra, then $j(\cD)$ is a coidempotent object in $\iNS$. The same assertion holds for $\cpt$, i.e., $j(\cpt)$ is a coidempotent object in $\iNS$.
\end{lem}

\begin{proof}
 Let $X$ stand for $\cD$ or $\cpt$. Since $X$ is an idempotent object in $\iCsep$, it becomes a coidempotent object in $\iCsep^\op$. Consequently, $j(X)$ becomes a coidempotent object in $\iNS$ (since $j:\iCsep^\op\functor\iNS$ is a fully faithful symmetric monoidal functor).
\end{proof}

\begin{lem} \label{idem}
 For any strongly self-absorbing $C^*$-algebra $\cD$, the stabilization $\Sigmat(\cD)$ is a coidempotent object in $\iNSp$. The same assertion holds for $\cpt$, i.e., $\Sigmat (\cpt)$ is a coidempotent object in $\iNSp$.
\end{lem}

\begin{proof}
 Since $\Sigmat:\iNS\map\iNSp$ is symmetric monoidal (see Theorem \ref{SMNSp}), the assertion follows from the previous Lemma.
\end{proof}

\noindent
Recall that a functor $R:\cC\map\cC$ is called a {\em colocalization} if $R:\cC\map R\cC$ is the right adjoint to the inclusion $R\cC\subseteq\cC$; in particular, the inclusion is the left adjoint to $R$ and hence preserves all small colimits. Owing to the fact that $\iNSp$ is closed symmetric monoidal (see Theorem \ref{SMNSp}), one may consider the colimit preserving endofunctor $-\otimes \Sigmat (A):\iNSp\functor\iNSp$ for any $A\in\iCsep^\op$. Often such functors are colocalizations.

\begin{prop} \label{Col}
 Let $A$ be any strongly self-absorbing $C^*$-algebra $\cD$ or $\cpt$. Then the functors $R_A:\iNS\functor\iNS$ and $R_{\Sigmat (A)} : \iNSp\functor\iNSp$ given by $R_A(X)=X\otimes j(A)$ and $R_{\Sigmat (A)}(X)= X\otimes\Sigmat (A)$ respectively are colocalization functors.
\end{prop}

\begin{proof}
The assertions follow from the dual of Proposition 4.8.2.4 of \cite{LurHigAlg}. 
\end{proof}

\begin{dis} \label{warning}
In \cite{MyNGH} the author called the groups $\NSH(\CC,-)$ (resp. $\NSH(-,\CC)$) the noncommutative stable homotopy (resp. noncommutative stable cohomotopy) groups. The terminology was motivated by the fact that $\NSH(\CC,-)$ is covariant and $\NSH(-,\CC)$ is contravariant. However, it was observed in \cite{MyNGH} that $\NSH(\CC,-)$ generalizes stable cohomotopy, whereas $\NSH(-,\CC)$ generalizes stable homotopy of finite pointed CW complexes. In order to align the theory with the terminology familiar to topologists, we rename them following Definition 3.2 of \cite{MyNSH} as follows: 
\beqn
\NSH(\CC,-) &=& \text{ noncommutative stable cohomotopy}\\
\NSH(-,\CC) &=& \text{ noncommutative stable homotopy}
\eeqn We also extend the terminology predictably to their graded versions.
\end{dis}

\subsection{(Co)localizations and purely infinite strongly self absorbing $C^*$-algebras}
The list of known examples of strongly self-absorbing $C^*$-algebras is rather limited. The list includes Cuntz algebras $\mathcal{O}_2$ and $\Oinf$, the Jiang--Su algebra $\mathcal{Z}$, UHF algebras of infinite type, and tensor products of $\Oinf$ with UHF algebras of infinite type. It follows from the results of Kirchberg that strongly self-absorbing $C^*$-algebras are either stably finite or purely infinite. In the purely infinite case Toms--Winter completely classified all strongly self-absorbing $C^*$-algebras satisfying UCT (Corollary page 4022 \cite{TomWin}), viz., they are $\mathcal{O}_2$, $\Oinf$ and tensor products of $\Oinf$ with UHF algebras of infinite type. We are particularly interested in the purely infinite ones since $ax+b$-semigroup $C^*$-algebras of number rings are all purely infinite (Corollary 8.2.11 of \cite{CunEchLi1}). Among the strongly self-absorbing purely infinite $C^*$-algebras $\Oinf$ plays a distinguished role in the classification program. The $C^*$-algebra $A\prot\Oinf$ is purely infinite for any $A\in\Csep$ \cite{KirRor}. Deviating slightly from the predictable pattern the colocalization of $\iNSp$ by the functor $R_{\Sigmat (\cD)}(-)=-\otimes\Sigmat (\cD)$ is denoted by $\iNSp[\cD^{-1}]$ (and not by $\iNSp[(\Sigmat (\cD))^{-1}]$). In what follows we are going to drop the object $\Sigmat (\cD)$ from the colocalization functor $R_{\Sigmat (\cD)}$ and denote it simply by $R$.

Thanks to Proposition \ref{Col} above one can study colocalizations of both $\iNS$ and $\iNSp$ with respect to a strongly self-absorbing $C^*$-algebra $\cD$ or $\cpt$. Recall that bivariant $\E$-theory is a bivariant homology theory of separable $C^*$-algebras that agrees with $\KK$-theory for all nuclear $C^*$-algebras. Hence it is considered to be quite computable.

\begin{prop} \label{Ksummand}
Let $A,B\in\Csep$ and $\cD$ be any purely infinite strongly self-absorbing $C^*$-algebra satisfying UCT. Then the bivariant noncommutative stable homotopy group $\hNSp[\cD^{-1}](R(\Sigmat (A)),R(\Sigmat (B)))$ contains $\E_0(B\prot\cD,A\prot\cD)$ as a natural summand.
\end{prop}

\begin{proof}
By construction there is a natural identification $$\hNSp[\cD^{-1}](R(\Sigmat (A)),R(\Sigmat (B)))\cong\hNSp(\Sigmat (A\prot\cD),\Sigmat (B\prot\cD)),$$ where we used the fact that $\Sigmat:\h\iCsep^\op\functor\hNSp$ is symmetric monoidal (see Theorem \ref{SMNSp}). Using Theorem 2.26 of \cite{MyNSH} we also deduce that $$\hNSp(\Sigmat (A\prot\cD),\Sigmat (B\prot\cD))\cong\NSH (B\prot\cD, A\prot\cD).$$ Thus it suffices to show that $\NSH(B\prot\cD, A\prot\cD)$ contains $\E_0(B\prot\cD,A\prot\cD)$ as a summand. Now consider the canonical composition of $*$-homomorphisms $\cpt\overset{i}{\map}\cD\overset{\theta}{\map}\cD\prot\cpt$. Note than any purely infinite strongly self-absorbing $C^*$-algebra $\cD$ admits a unital embedding $\Oinf\overset{\iota}{\hookrightarrow}\cD$ \cite{WinZStable}. Here $i:\cpt\map\cD$ maps $e_{ij}$ to $\iota(s_i)\iota(s_j)^*$, where $\{s_i\}_{i\in\NN}$ is the standard set of generators of $\Oinf$, and $\theta:\cD\map\cD\prot\cpt$ is the corner embedding $a\mapsto a\otimes e_{11}$. Tensoring the diagram with $A\prot\cD$ and applying $\NSH(B\prot\cD,-)$ leads to the following diagram \beqn \NSH(B\prot\cD,A\prot\cD\prot\cpt)\overset{i}{\map}\NSH(B\prot\cD,A\prot\cD\prot\cD)\overset{\theta}{\map}\NSH(B\prot\cD,A\prot\cD\prot\cD\prot\cpt).\eeqn

Observe that for any $E,F\in\Csep$ there is a natural map $\NSH(E,F)\map\E_0(E,F)$, which becomes an isomorphism as soon as $F$ is stable (see Theorem 4.1.1. of \cite{ThomThesis}; also \cite{HigETh}). Therefore, the above diagram can be naturally identified with \beq \label{split} \E_0(B\prot\cD,A\prot\cD\prot\cpt)\overset{i}{\map}\NSH(B\prot\cD,A\prot\cD\prot\cD)\overset{\theta}{\map}\E_0(B\prot\cD,A\prot\cD\prot\cD\prot\cpt).\eeq Since $\cD$ satisfies UCT, the composition $\cD\prot\cpt\overset{i}{\map}\cD\prot\cD\overset{\theta}{\map}\cD\prot\cD\prot\cpt$ produces a $\E$-equivalence. Therefore, the composition $A\prot\cD\prot\cpt\overset{i}{\map}A\prot\cD\prot\cD\overset{\theta}{\map}A\prot\cD\prot\cD\prot\cpt$ is also an $\E$-equivalence whence the composition in diagram \eqref{split} is an isomorphism. Using the self-absorbing property of $\cD$, i.e., $\cD\cong\cD\prot\cD$, we conclude that the map $$\NSH(B\prot\cD,A\prot\cD)\cong\NSH(B\prot\cD,A\prot\cD\prot\cD)\overset{\theta}{\map}\E_0(B\prot\cD,A\prot\cD\prot\cD\prot\cpt)\cong\E_0(B\prot\cD,A\prot\cD\prot\cpt)$$ is split surjective. Finally, due to $C^*$-stability of bivariant $\E$-theory we may naturally identify $\E_0(B\prot\cD,A\prot\cD\prot\cpt)\cong\E_0(B\prot\cD,A\prot\cD)$, i.e., $\E_0(B\prot\cD,A\prot\cD)$ is a split summand of $$\NSH (B\prot\cD, A\prot\cD)\cong\hNSp[\cD^{-1}](R(\Sigmat (A)),R(\Sigmat (B))).$$
\end{proof}

\begin{cor} \label{KsummandCor}
The proof of Proposition \ref{Ksummand} actually shows that the noncommutative stable cohomotopy groups of any $\cD$-stable separable $C^*$-algebra contains its topological $\K$-theory groups as natural summands. If $\cD=\Oinf$ then one may replace $\E_0(B\prot\cD,A\prot\cD)$ in the above Proposition by $\E_0(B,A)$ due to $\Oinf$-stability of bivariant $\E$-theory in both variables. 
\end{cor}

\begin{rem} \label{Einfty}
Recall from Lemma \ref{idem} that $\Sigmat(\cpt)$ is a coidempotent object in $\iNSp$. Thus one may construct the smashing colocalization $R_{\cpt}:\iNSp\functor\iNSp[\cpt^{-1}]$. We set $\mathtt{E_\infty}:=\iNSp[\cpt^{-1}]^\op$ and it is our proposed stable $\infty$-categorical model for bivariant $\E$-theory. Arguments parallel to those in Section 2 of \cite{MyColoc} will show that $\iNSp[\cpt^{-1}]$ is a closed symmetric monoidal and compactly generated stable $\infty$-category; moreover, there is an exact fully faithful functor from the bivariant $\E$-theory category of separable $C^*$-algebras to $\h\mathtt{E_\infty}$.
\end{rem}

\subsection{Bivariant connective $\E$-theory}
Earlier we had outlined the construction of bivariant connective $\E$-theory in the setting of $\infty$-categories (see Remark 2.29 of \cite{MyNSH}). We furnish the details here. Consider the set of maps $X=\{M_2(A)\map A\,|\, A\in\iCsep^\op\}$ induced by the corner embeddings in $\iCsep^\op$. In \cite{ThomThesis} Thom constructed the bivariant connective $\E$-theory category as the Verdier quotient $\NSH\functor\NSH[(X^\op)^{-1}]$. Following \cite{ThomThesis} we denote the Verdier quotient, which is the bivariant connective $\E$-theory category, by $\bu$. There is a symmetric monoidal colimit preserving suspension spectrum functor $\Sigmat:\iNS\functor\iNSp$ (see Theorem \ref{SMNSp}). Via $\iCsep^\op\overset{j}{\hookrightarrow}\iNS\overset{\Sigmat}{\functor}\iNSp$ from $X$ we obtain a set of maps $X'$ between compact objects in $\iNSp$.

\begin{defn}\label{bu} We denote the opposite of the codomain of the accessible localization $\iNSp\overset{L_{X'}}{\map} X'^{-1}\iNSp$ by $\Ecn$ and this is the stable $\infty$-categorical version of bivariant connective $\E$-theory \cite{ThomThesis}. We denote the composite functor $\iNS^\op\overset{(\Sigmat)^\op}{\functor}{\iNSp^\op}\overset{(L_{X'})^\op}{\functor}\Ecn$ by $\kc$. See Theorem \ref{buTop} and Example \ref{buName} below for a justificaton of this terminology.\end{defn}

\begin{rem} \label{ECtoE}
 There is a canonical functor $X'^{-1}\iNSp\functor\iNSp[\cpt^{-1}]$ owing to the fact that the colocalization $R_{\cpt}:\iNSp\functor\iNSp[\cpt^{-1}]$ is a colimit preserving functor between presentable $\infty$-categories that sends the maps in $X'$ to equivalences. Taking the opposite of the functor $X'^{-1}\iNSp\functor\iNSp[\cpt^{-1}]$ we get a canonical functor $\Ecn\functor\mathtt{E_\infty}$ (see Remark \ref{Einfty}). 
\end{rem}

\begin{prop}
 The localization $\iNSp\overset{L_{X'}}{\map} (\Ecn)^\op$ is a symmetric monoidal colimit preserving functor between symmetric monoidal and compactly generated stable $\infty$-categories.
\end{prop}

\begin{proof}
 Since $\iNSp$ is compactly generated and $L_{X'}$ is an accessible localization, such that the domains and codomains of the maps in $X'$ are all compact, the stable $\infty$-category $\Ecn$ is compactly generated and $L_{X'}$ is colimit preserving. Moreover, $\iNSp$ is a closed symmetric monoidal $\infty$-category (see Theorem \ref{SMNSp}). Thus it suffices to show that $L_{X'}$ is a symmetric monoidal localization. For every $\Sigmat(M_2(A))\map \Sigmat(A)$ in $X'$ the map $\Sigmat(M_2(A))\otimes\Sigmat(Y)\map\Sigmat(A)\otimes\Sigmat(Y)$ belongs to the strongly saturated collection of morphisms in $\iNSp$ generated by $X'$. Arguing as in the proof of Theorem 1.4 of \cite{MyColoc} one deduces that $L_{X'}$ is a symmetric monoidal functor between closed symmetric monoidal stable $\infty$-categories.
\end{proof}

\begin{cor} \label{SMbu}
The functor $\kc=(L_{X'})^\op\circ(\Sigmat)^\op:\iNS^\op\functor\Ecn$ is symmetric monoidal. 
\end{cor}

\noindent
The following Theorem demonstrates that Definition \ref{bu} is appropriate. 

\begin{thm} \label{buTop}
 There is a fully faithful exact functor $\bu\hookrightarrow\hEcn$ thereby showing that $\bu$ is a topological triangulated category.
\end{thm}

\begin{proof}
Recall from Theorem 2.26 of \cite{MyNSH} that there is a fully faithful exact functor $\pi:\NSH\functor\hNSp^\op$. We consider its opposite $\pi^\op:\NSH^\op\functor\hNSp$, which is also fully faithful and whose image lies inside the compact objects of $\hNSp$. Since $\pi^\op(X) = X'$ by construction there is the following commutative diagram:
 
 \beqn
 \xymatrix{
 \NSH^\op\ar[r]^{\pi^\op}\ar[d]_V & \hNSp\ar[d]^{L_{X'}} \\
 \bu^\op \ar@{-->}[r] & \hEcn^\op. 
 }\eeqn The dashed functor making the above diagram commute exists because the triangulated category $\ker(V)$ is generated by $\{\mathrm{cone}(f)\,|\, f\in X\}$ and $\ker(L_{X'})$ is compactly generated by $\pi^\op(\ker(V))$. It follows from Theorem 7.2.1(3) and Lemma 4.7.1 of \cite{KraLoc} that the dashed functor is fully faithful. Taking its opposite furnishes the desired fully faithful exact functor $\bu\hookrightarrow\hEcn$. An argument similar to Theorem 2.27 of \cite{MyNSH} shows that $\bu$ is topological.
\end{proof}

\begin{ex} \label{buName}
Let $(X,x)$ and $(Y,y)$ be two finite pointed CW complexes. Then $$\bu(\C(X,x),\C(Y,y)) \cong \colim_m\, [\Sigma^m Y, \mathbf{ku}_m \wedge X],$$ where $\mathbf{ku}$ denotes the connective $\K$-theory spectrum (see Theorem 4.2.1 of \cite{ThomThesis}).
\end{ex}

Thanks to Corollary \ref{SMbu} we conclude that if $A$ is a coidempotent object in $\iCsep^\op$ (and hence an idempotent object in $\iCsep$), then $\kc(A)$ is an idempotent object in $\Ecn$. We denote the smashing localization of $\Ecn$ with respect to such an idempotent object $\kc(A)$ by $\Ecn[A^{-1}]$. From Lemma \ref{SSAidem} we know that every strongly self-absorbing $C^*$-algebra $\cD$ is an idempotent object in $\iCsep$. For brevity we denote the composite functor $\iNS^\op\overset{\kc}{\map}\Ecn\map\Ecn[\cD^{-1}]$ also by $\kc$ in the sequel. Since $\Ecn\map\Ecn[\cD^{-1}]$ is a smashing localization the functor $\kc:\iNS^\op\map\Ecn[\cD^{-1}]$ is also symmetric monoidal.

\begin{thm} \label{OColoc}
For any $A,B\in\Csep$ there is a natural isomorphism $$\hEcn[\Oinf^{-1}](\kc(A),\kc(B))\cong \E_0(A,B).$$
\end{thm}

\begin{proof}
 By construction there is a natural identification $$\hEcn[\Oinf^{-1}](\kc (A),\kc (B))\cong\hEcn(\kc (A\prot\Oinf),\kc (B\prot\Oinf)),$$ where we used the fact that $\h\iCsep\functor\iNS^\op\overset{\kc}{\functor}\hEcn$ is symmetric monoidal (see Corollary \ref{SMbu}). Using Theorem \ref{buTop} we also deduce that $$\hEcn(\kc(A\prot\Oinf),\kc (B\prot\Oinf))\cong\bu(A\prot\Oinf, B\prot\Oinf).$$ Thus it suffices to show that $\bu(A\prot\Oinf, B\prot\Oinf)\cong\E_0(A,B)$. Now consider again the composition of $*$-homomorphisms $\cpt\overset{i}{\map}\Oinf\overset{\theta}{\map}\Oinf\prot\cpt$ as in the proof of Proposition \ref{Ksummand}. Tensoring the diagram with $B$ and applying the functor $\bu(A,-)$ we get $$\bu(A,B\prot\cpt)\overset{i}{\map}\bu(A,B\prot\Oinf)\overset{\theta}{\map}\bu(A,B\prot\Oinf\prot\cpt).$$ From Lemma 4.2.4 of \cite{ThomThesis} we know that for $C,D\in\Csep$ we have $\bu(C,D)\cong\E_0(C,D)$ naturally whenever $D$ is stable. Hence the above composition of homomorphisms can be naturally identified with $$\E_0(A,B\prot\cpt)\overset{i}{\map}\bu(A,B\prot\Oinf)\overset{\theta}{\map}\E_0(A,B\prot\Oinf\prot\cpt).$$ This composition is an $\E$-equivalence, since $\theta\circ i:\cpt\map\Oinf\prot\cpt$ is an $\E$-equivalence. It follows that $\theta$ is surjective. Now consider a different comoposition of $*$-homomorphisms $\Oinf\overset{\theta}{\map}\Oinf\prot\cpt\overset{\kappa}{\map}\Oinf$, where $\kappa(a\otimes e_{ij})= s_i a s_j^*$. The composite $*$-endomorphism $\kappa\circ\theta:\Oinf\map\Oinf$ is inner, i.e.,  $a\mapsto s_1 a s_1^*$ and tensoring the diagram with unital $B$ we again get a composite inner $*$-endomorphism of $B\prot\Oinf$. Now applying the matrix stable functor $\bu(A,-)$ and using Proposition 3.16 of \cite{CunMeyRos} we find that the composition $$\bu(A,B\prot\Oinf)\overset{\theta}{\map}\bu(A,B\prot\Oinf\prot\cpt)\overset{\kappa}{\map}\bu(A,B\prot\Oinf)$$ is an isomorphism. Applying Lemma 4.2.4 of \cite{ThomThesis} again we can naturally identify the above composition with $$\bu(A,B\prot\Oinf)\overset{\theta}{\map}\E_0(A,B\prot\Oinf\prot\cpt)\overset{\kappa}{\map}\bu(A,B\prot\Oinf),$$ whose composition is an isomorphism. This implies that $\theta$ is also injective. Thus we have proven that $\theta:\bu(A,B\prot\Oinf)\overset{\sim}{\map}\E_0(A,B\prot\Oinf\prot\cpt)$ is an isomorphism. Finally, using $C^*$-stability and $\Oinf$-stability of bivariant $\E$-theory one has $\E_0(A,B\prot\Oinf\prot\cpt)\cong\E_0(A,B)$, which is also a natural isomorphism. The proof extends to all (nonunital) $B$ by a simple excision argument.\end{proof}

\begin{rem} \label{OComp}
 An inspection of the proof of Theorem \ref{OColoc} demonstrates that actually a stronger result holds, viz., $$\hEcn(\kc(A),\kc (B\prot\Oinf))\cong\E_0(A,B)$$ for any $A,B\in\Csep$.
\end{rem}

\begin{cor} \label{MyComparison}
 The nonconnective algebraic $\K$-theory of $\Oinf$-stable separable $C^*$-algebras factors through the essential image of $\h\iCsep\functor\h\iNS^\op\overset{}{\functor}\hEcn[\Oinf^{-1}]$.
\end{cor}

\begin{proof}
 It was shown in \cite{CorPhi,MyComparison} that the nonconnective algebraic $\K$-theory of $\Oinf$-stable $C^*$-algebras agrees naturally with their topological $\K$-theory. The assertion now follows since topological $\K$-theory, which is naturally isomorphic to $\E$-theory, has the desired property.
\end{proof}

\noindent
A more useful version of the above result is proven below (see Theorem \ref{facKK}). Let $\cQ$ denote any UHF algebra of infinite type, so that $\Oinf\prot\cQ$ is a purely infinite strongly self-absorbing $C^*$-algebra; in fact, any such $C^*$-algebra, that additionally satisfies UCT, is up to isomorphism $\cO_2$, $\Oinf$, or of the form $\Oinf\prot\cQ$ (see Corollary on page 4022 of \cite{TomWin}).

\begin{thm} \label{UHF}
For any $A,B\in\Csep$ there is a natural isomorphism $$\hEcn[(\Oinf\prot\cQ)^{-1}](\kc (A),\kc(B))\cong\E_0(A\prot\cQ,B\prot\cQ).$$
\end{thm}

\begin{proof}
As before we first observe that $$\hEcn[(\Oinf\prot\cQ)^{-1}](\kc(A),\kc(B))\cong\hEcn(\kc(A\prot\Oinf\prot\cQ),\kc (B\prot\Oinf\prot\cQ)).$$ Arguing as in the previous Theorem one then proves that $$\hEcn(\kc(A\prot\Oinf\prot\cQ),\kc (B\prot\Oinf\prot\cQ))\cong\E_0(A\prot\cQ,B\prot\cQ).$$ 
\end{proof}

\begin{ex}
The stable $\infty$-category consisting of the compact objects of $\Ecn[(\Oinf\prot\cQ)^{-1}]$ constitutes an $\infty$-categorical model for the opposite of rationalized bivariant $\E$-theory category. Indeed, it is well known that tensoring with the universal UHF algebra rationalizes $\E$-theory. For instance, we have the following sequence of isomorphisms: \beqn
\E_i(\Oinf\prot\cQ,A\prot\cQ)&\cong&\E_i(\Oinf\prot\cQ,A\prot(\Oinf\prot\cQ)) \text{ [use $\Oinf$-stability in bivariant $\E$-theory]}\\
&\cong& \E_i(A\prot\Oinf\prot\cQ) \text{ [use Theorem in Section 3 of \cite{DadWin} with $\cD=\Oinf\prot\cQ$]}\\
&\cong& \E_i(A\prot\cQ) \text{ [use $\Oinf$-stability in $\E$-theory]}\\
&\cong& \E_i(A)\otimes_\ZZ \QQ \text{ [use K{\"u}nneth formula in $\E$-theory]}
\eeqn for $i=0,1$. Thus this localization can be viewed as an $\infty$-categorical model for the noncommutative Chern--Connes character in bivariant connective $\E$-theory.
\end{ex}

\begin{thm} \label{O2}
 For any $A,B\in\Csep$ there is a natural isomorphism $$\hEcn[\cO_2^{-1}](\kc(A),\kc(B))\cong 0.$$
\end{thm}

\begin{proof}
 Once again arguing as in Theorem \ref{OColoc} we observe that $$\hEcn[\cO_2^{-1}](\kc(A),\kc(B))\cong\bu(A\prot\cO_2,B\prot\cO_2).$$ Since $\cO_2$ is a simple and properly infinite $C^*$-algebra one can again find a diagram in $\Csep$ $$\cO_2\overset{\theta}{\map}\cO_2\prot\cpt\overset{\kappa}{\map}\cO_2,$$ such that the composition is an inner endormorphism. Indeed, by Proposition 1.1.2 of \cite{RorClass} there is a sequence $\{s_n\}_{n=1}^\infty$ of partial isometries in $\cO_2$, such that $s_i^*s_i = 1$ for all $i$ and the range projections $s_is_i^*$ are mutually orthogonal subprojections of $1$; thus we may choose $\theta(a)=a\otimes e_{11}$ and $\kappa(a\otimes e_{ij}) = s_i a s_j^*$. Tensoring the diagram with unital $B$ we get another diagram $$B\prot\cO_2\map B\prot\cO_2\prot\cpt\map B\prot\cO_2,$$ such that the composition is again an inner endormorphism. Applying the matrix stable functor $\bu(A\prot\cO_2,-)$ to the above diagram we find that $\bu(A\prot\cO_2, B\prot\cO_2)$ is a summand of $\bu(A\prot\cO_2, B\prot\cO_2\prot\cpt)\cong\E_0(A\prot\cO_2,B\prot\cO_2\prot\cpt).$ Thus it suffices to show that the group $\E_0(A\prot\cO_2, B\prot\cO_2)\cong\E_0(A\prot\cO_2, B\prot\cO_2)$ vanishes. Since $\cO_2$ is $\KK$-contractible, so is $A\prot\cO_2$ and hence it satisfies UCT. Thus one may identify $\E_0(A\prot\cO_2,B\prot\cO_2)\cong\KK_0(A\prot\cO_2,B\prot\cO_2)$ and the group  $\KK_0(A\prot\cO_2,B\prot\cO_2)$ evidently vanishes. The general case, i.e., when $B$ is nonunital, is treated again by an excision argument.
\end{proof}

\begin{rem}
 The stable $\infty$-category $\hEcn[\cO_2^{-1}]$ is compactly generated by the suspensions and desuspensions of $\kc(A)$ for all $A\in\iCsep$. The above Theorem shows that the generators vanish whence the whole stable $\infty$-category vanishes.
\end{rem}

\subsection{Connective $\E$-theory of $ax+b$-semigroup $C^*$-algebras of number rings} 
Number rings are central objects of study in number theory. One can associate an $ax+b$-semigroup $C^*$-algebra with any number ring that possesses very intriguing structure \cite{CunDenLac}. Given the interest generated by such $C^*$-algebras ascertaining their (co)homological invariants seems to be an important task. For any $A\in\Csep$ we call $\bu(\CC,A)$ the {\em connective $\E$-theory} group of $A$. The corresponding graded version can be obtained by (de)suspensions.

For a countable integral domain $R$ with vanishing Jacobson radical (which is, in addition, not a field) the left regular $ax+b$-semigroup $C^*$-algebra $C^*_\lambda(R\rtimes R^\times)$ is $\Oinf$-stable, i.e., $C^*_\lambda(R\rtimes R^\times)\prot\Oinf\cong C^*_\lambda(R\rtimes R^\times)$ (see Theorem 1.3 of \cite{LiAxb}). Cuntz--Echterhoff--Li computed the topological $\K$-theory of such $ax+b$-semigroup $C^*$-algebras in \cite{CunEchLi1} as follows: 
\beq \label{Ktheory}
\K_*(C^*_\lambda(R\rtimes R^\times))\cong \underset{{[Y]\in G\setminus\mathcal{I}}}{\oplus} \K_*(C^*(G_Y)),
\eeq where $\mathcal{I}$ is the set of fractional ideal of $R$, $G=K\rtimes K^\times$, and $G_Y$ is the stabilizer of $Y$ under the $G$-action on $\mathcal{I}$. The orbit space $G\setminus \mathcal{I}$ can be identified with the ideal class group of $K$. The $\K$-theory of the group $C^*$-algebras appearing as summands in the above formula can be computed by the Baum--Connes conjecture.

\begin{thm} \label{AX+B}
 The connective $\E$-theory of the left regular $ax+b$-semigroup $C^*$-algebra of the ring of integers $R$ of a number field $K$ is $2$-periodic and explicitly given by $$\bu(\CC,C^*_\lambda(R\rtimes R^\times))\cong\underset{{[Y]\in G\setminus\mathcal{I}}}{\oplus} \K_0(C^*(G_Y)).$$ and $$\bu(\CC,\Sigma C^*_\lambda(R\rtimes R^\times))\cong\underset{{[Y]\in G\setminus\mathcal{I}}}{\oplus} \K_1(C^*(G_Y)).$$
\end{thm}

\begin{proof}
Since $C^*_\lambda(R\rtimes R^\times)$ is $\Oinf$-stable, there is an identification of the connective $\E$-theory group $\bu(\CC,C^*_\lambda(R\rtimes R^\times))\cong\bu(\CC,C^*_\lambda(R\rtimes R^\times)\prot\Oinf)$. By Remark \ref{OComp} we conclude that  $\bu(\CC,C^*_\lambda(R\rtimes R^\times)\prot\Oinf)\cong\E_0(\CC,C^*_\lambda(R\rtimes R^\times))$. One may identify the $\E$-theory of $C^*_\lambda(R\rtimes R^\times)$ naturally with its topological $\K$-theory. The results now follow from Equation \eqref{Ktheory} (the second one after suspension of $C^*_\lambda(R\rtimes R^\times)$).
\end{proof}

\begin{rem} \label{ax+bsum}
Owing to the $\Oinf$-stability of $C^*_\lambda(R\rtimes R^\times)$, Corollary \ref{KsummandCor} asserts that its noncommutative stable cohomotopy contains its topological $\K$-theory as a summand.
\end{rem}

\section{Nonconnective $\KQ$-theory and $\infty$-categorical topological $\TT$-duality}

We work exclusively in the category of nonunital (not necessarily unital) $k$-algebras, denoted by $\Alg_k$, where $k$ is a field of characteristic zero as in \cite{KQ}. The morphisms in $\Alg_k$ are $k$-algebra homomorphisms. At certain places in the sequel we are admittedly sloppy regarding size issues; however, as is common in $\K$-theory there will always be a small skeleton that comes to our rescue.

\subsection{Stable $\infty$-category valued noncommutative motives} For any $k$-algebra $A$ let $\tilde{A}$ denote its $k$-unitization with underlying $k$-linear space $A\oplus k$ and multiplication $(a,\lambda)(a',\lambda')= (aa' + \lambda a' + \lambda' a, \lambda\lambda')$. The category $\Mod(\tilde{A})$ is an abelian category. In \cite{KQ} we considered the following differential graded category $\hpfd({A})$: its objects are cochain complexes of right $\tilde{A}$-modules, such that each such $Y$ is homotopy equivalent to a complex $X$ satisfying 
\begin{enumerate}
 \item $X$ is homotopy equivalent to a strictly perfect complex,
 \item the canonical map $X\otimes_{\tilde{A}} A\map X$ is a homotopy equivalence.
\end{enumerate}

A $k$-linear cochain complex worth of morphisms between two such objects is obtained in a standard manner. We are going to consider an $\infty$-categorical variant of $\hpfd$-construction. There is a {\em differential graded nerve} $\N_\dg$ of a differential graded category (see Construction 1.3.1.6 of \cite{LurHigAlg}). For a differential graded category $\cC$ as a simplicial set $\N_\dg(\cC)$ can be described as follows: 

\begin{itemize}
 \item the $0$-simplices are the objects of $\cC$,
 \item the $1$-simplices are $\cup_{X,Y\in\mathrm{Ob}(\cC)} \{f\in\cC(X,Y)^0\,|\, df=0\}$.
\end{itemize}

In order to get an idea about the higher simplices let us note that the differential graded nerve $\N_\dg(\cC)$ (of a homologically graded differential graded category $\cC$) is obtained by applying the homotopy coherent nerve to a Kan complex enriched category constructed out of $\cC$. The Kan complex enriched category is obtained by first applying the truncation $\tau_{\geqslant 0}$ to the mapping complexes in $\cC$ and then applying the Dold--Kan construction.

Using the above construction we manufacture an $\infty$-category $\C_\infty(\tilde{A})$ out of the differential graded category of cochain complexes of $\Mod(\tilde{A})$, which turns out to be a stable $\infty$-category (see Proposition 1.3.2.10 of \cite{LurHigAlg}). By construction the objects (or $0$-simplices) of the differential graded nerve $\C_\infty(\tilde{A})$ are complexes of right $\tilde{A}$-modules. Let us set $\hpf({A})$ to be the stable $\infty$-subcategory of $\C_\infty(\tilde{A})$ spanned by the objects of $\hpfd({A})$. 

\begin{rem} \label{homotopyCompatible}
There is an isomorphism of homotopy categories $\h\hpf({A})\cong\h(\hpfd(A)):=\H^0(\hpfd({A}))$ (see Remark 1.3.1.11 of \cite{LurHigAlg}).
\end{rem}

\begin{rem}
 In the world of algebra the convention is to grade complexes cohomologically, whereas, in topology one considers typically homologically graded complexes. In \cite{KQ} the differential graded complexes were cohomologically graded as it built upon the formalism of \cite{KelDG}, whereas in \cite{LurHigAlg} they are homologically graded. The passage between the two is not too difficult (see, for instance, Definitions 3.1.6 and 3.3.1 of \cite{Faonte}).
\end{rem}

Let $\SSet$ denote the category of simplicial sets with the Joyal model structure, whose fibrant objects are precisely the $\infty$-categories. Recall that $\catinf$ is $\infty$-category of (small) $\infty$-categories, which is obtained by applying the homotopy coherent nerve to the Kan complex enriched category $\mathtt{Cat}_\infty^\Delta$, whose objects are (small) $\infty$-categories and the mapping space between $\cC$ and $\cD$ is given by the largest Kan complex contained in $\Fun(\cC,\cD)=\Map_{\SSet}(\cC,\cD)$. Let $\catex$ denote the $\infty$-subcategory of (small) stable $\infty$-categories with exact functors.

\begin{prop} \label{infFun}
 The association $A\mapsto\hpf({A})$ produces a functor $\N(\Alg_k)\functor\catex$.
\end{prop}

\begin{proof}
 Clearly the construction $A\mapsto\tilde{A}$ is functorial. It sends any $k$-algebra homomorphism to a unital $k$-algebra homomorphism. Sending $\tilde{A}$ to the differential graded category of cochain complexes of right $\tilde{A}$-modules is also functorial up to a coherent natural isomorphism: $\tilde{A}\map\tilde{B}$ induces the map $-\otimes_{\tilde{A}}\tilde{B}$ between the differential graded categories. Application of the differential graded nerve produces the $\infty$-category $\C_\infty(\tilde{A})$ and lands inside $\SSet$ (see Proposition 1.3.1.20 of \cite{LurHigAlg}). Now applying the homotopy coherent nerve construction we get a functor $\N(\Alg_k)\functor\catinf$. Thus we have demonstrated that the association $A\mapsto\C_\infty(\tilde{A})$ produces a functor $\N(\Alg_k)\functor\catinf$. 
 
 Let us now verify that $-\otimes_{\tilde{A}}\tilde{B}:\C_\infty(\tilde{A})\map\C_\infty(\tilde{B})$ restricts to an arrow $\hpf({A})\map\hpf({B})$. It is easy to see that $-\otimes_{\tilde{A}}\tilde{B}$ sends strictly perfect complexes of right $\tilde{A}$-modules to strictly perfect complexes of right $\tilde{B}$-modules. The functor also preserves homotopy equivalences whence condition (1) above is preserved. We need to now check that the canonical map $Y\otimes_{\tilde{A}}\tilde{B}\otimes_{\tilde{B}} B\map Y\otimes_{\tilde{A}}\tilde{B}$ is a homotopy equivalence. Since the functor $-\otimes_{\tilde{A}}\tilde{B}$ preserves homotopy equivalences, we may assume that $Y$ is strictly perfect. Since $Y\otimes_{\tilde{A}}\tilde{B}\otimes_{\tilde{B}} B \cong Y\otimes_{\tilde{A}} B$ it suffices to show that $Y\otimes_{\tilde{A}} B\map Y\otimes_{\tilde{A}}\tilde{B}$ is a homotopy equivalence. Tensoring the short exact sequence of $\tilde{A}$-modules $0\map B\map\tilde{B}\map k\map 0$ with the strictly perfect complex $Y$, we are reduced to showing $Y\otimes_{\tilde{A}} (\tilde{A}/A)\cong Y\otimes_{\tilde{A}} k$ is acyclic. Since $Y\in\hpf({A})$ the canonical map $Y\otimes_{\tilde{A}} A\map Y$ is a homotopy equivalence. It follows that $Y\otimes_{\tilde{A}} k$ is acyclic. Consequently, we have a functor $\hpf:\N(\Alg_k)\functor\catinf$.

 By construction $\hpf({A})$ is a stable $\infty$-category. Thus it suffices to show that the functor $-\otimes_{\tilde{A}}\tilde{B}:\hpf({A})\functor\hpf({B})$ is exact. The homotopy cofiber sequences in the differential graded category $\hpfd({A})$ are equivalent to short exact sequences, that are split exact in each degree. They produce the cofiber sequences in the stable $\infty$-category $\hpf({A})$, which are clearly preserved by $-\otimes_{\tilde{A}}\tilde{B}:\hpf({A})\functor\hpf({B})$ whence the assertion follows.
\end{proof}

Let us recall from \cite{BluGepTab} that a diagram $\cA\map\cB\map\cC$ in $\catex$ is called {\em exact} if the sequence of stable presentable $\infty$-categories $\Ind(\cA)\map\Ind(\cB)\map\Ind(\cC)$ is exact, i.e., the composite is trivial, the functor $\Ind(\cA)\map\Ind(\cB)$ is fully faithful, and the canonical map $\Ind(\cB)/\Ind(\cA)\map\Ind(\cC)$ is an equivalence. Note that in \cite{BluGepTab} the treatment is more general as the notion of exactness is considered in $\mathtt{Cat_\infty^{ex(\kappa)}}$ for any regular cardinal $\kappa$, i.e., the $\infty$-category of $\kappa$-cocomplete small stable $\infty$-categories and $\kappa$-small colimit preserving functors.

\begin{lem} \label{homfib}
Let $0\map A\map B\map C\map 0$ be a short exact sequence in $\Alg_k$ with $B,C$ unital and $A^2=A$. If, in addition, $A$ is a flat $B$-module, then the induced diagram in $\catex$  $\hpf({A})\map\hpf(B)\map\hpf(C)$ is exact.
\end{lem}

\begin{proof}
 Lemma 2.14 of \cite{KQ} shows that the diagram $\hpfd({A})\map\hpfd(B)\map\hpfd(C)$ is an exact sequence of differential graded categories. This amounts to saying that the associated sequence of triangulated categories $\H^0(\hpfd(A))\map\H^0(\hpfd(B))\map\H^0(\hpfd(C))$ is exact, i.e., the composite is trivial, the functor $\H^0(\hpfd(A))\map\H^0(\hpfd(B))$ is fully faithful, and the canonical map $\H^0(\hpfd(B))/\H^0(\hpfd(A))\map\H^0(\hpfd(C))$ is an equivalence after idempotent completion. Proposition 5.15 of \cite{BluGepTab} says that a diagram $\cA\map\cB\map\cC$ in $\catex$ is exact if and only if the sequence of triangulated categories $\h\cA\map\h\cB\map\h\cC$ is exact, i.e., the composite is trivial, the functor $\h\cA\map\h\cB$ is fully faithful, and the canonical map $\h\cB/\h\cA\map\h\cC$ is an equivalence after idempotent completion. The assertion now follows from Remark \ref{homotopyCompatible} above.
\end{proof}

\begin{rem}
 We have used Lemma 2.14 of \cite{KQ} in the above argument, which asserts that if $0\map A\map B\map C\map 0$ is a short exact sequence in $\Alg_k$ with $B,C$ unital and $A^2=A$, then $\hpfd({A})\map\hpfd(B)\map\hpfd(C)$ is a short exact sequence of differential graded categories. Unfortunately, the assumption that $A$ is a flat $B$-module was not mentioned in Lemma 2.14 of \cite{KQ}. Together with the assumption $A^2=A$ this hypothesis implies that the homomorphism $B\overset{}{\map} C$ is a {\em homological epimorphism}, i.e., $C\otimes_B C\cong C$ and $\mathrm{Tor}^{B}_n(C,C) = 0 $ for $n\geqslant 1$. These hypotheses are satisfied by any short exact sequence of $C^*$-algebras (see the Example \ref{homEpi} below). 
\end{rem}

\begin{ex} \label{homEpi}
 It is known that every $C^*$-algebra is {\em universally flat} (see Theorem 2 of \cite{WodFunAna}); indeed, Theorem B of \cite{WodResolution} shows that any Banach algebra with a bounded approximate unit is universally flat and every $C^*$-algebra has a bounded approximate unit. It follows that, if $0\map A\map B\map C\map 0$ is short exact sequence of $C^*$-algebras, then $A^2=A$ (Cohen--Hewitt Factorization Theorem) and $A$ is a flat $B$-module. 
\end{ex}

In \cite{KQ} the $\KQ$-theory of $A\in\Alg_k$ was defined to be the (connective) algebraic $\K$-theory of the $k$-linear differential graded category $\hpfd(A)$. Any differential graded category $\cC$ has an underlying ordinary category $\cC'$ with morphisms given by ${\cC'}(X,Y)=\{f\in\cC(X,Y)^0\;|\; df =0\}$. The underlying category of the differential graded category $\hpfd(A)$ will be denoted by $\hpfw(A)$. There is a Waldhausen category structure on $\hpfw(A)$. One way to see this is as follows: The category of unbounded cochain complexes of $\tilde{A}$-modules $\Ch(\tilde{A})$ admits a model structure with cohomology isomorphisms as weak equivalences and degreewise epimorphisms as fibrations (see Theorem 2.3.11 of \cite{Hov}). A map $i:X\map Y$ in $\Ch(\tilde{A})$ is a cofibration if it is a degreewise split monomorphism with cofibrant cokernel (see Proposition 2.3.9 of \cite{Hov}). The subcategory of perfect complexes $\Perf(\tilde{A})$, which are the compact objects in $\Ch(\tilde{A})$ \cite{BokNee}, is a complete Waldhausen subcategory of the model category $\Ch(\tilde{A})$ in the sense of \cite{DugShi}. The category $\hpfw(A)$ is the full subcategory of $\Perf(\tilde{A})$ consisting of cochain complexes $X$ that are homotopy equivalent to strictly perfect complexes and satisfy $X\otimes_{\tilde{A}} A\map X$ is a homotopy equivalence. The strictly perfect complexes of $\tilde{A}$-modules are cofibrant in the model category $\Ch(\tilde{A})$ and the weak equivalences between such complexes are precisely the homotopy equivalences. One can now verify that the Waldhausen category structure on $\Perf(\tilde{A})$ restricts to a Waldhausen category structure on $\hpfw(A)$. Summarising, we have
 
\begin{lem} \label{Waldstr}
The category $\hpfw(A)$ is a Waldhausen subcategory of the model category $\Ch(\tilde{A})$ and the canonical inclusion $\hpfw(A)\hookrightarrow\Perf(\tilde{A})$ is Waldhausen exact.
\end{lem}

Let $\mathtt{Wald}$ denote the category of small Waldhausen categories with Waldhausen exact functors. One may apply the Waldhausen (connective) $\K$-theory functor $\bK^w:\mathtt{Wald}\functor\Sp$ to $\hpfw(A)$ to define its $\K$-theory. In Lemma 2.12 of \cite{KQ} we showed that $\KQ_i(A)\cong\pi_i(\bK^w(\hpfw(A)))$. Using the material from Section 7 of \cite{BluGepTab} we can define the connective $\K$-theory of the stable $\infty$-category $\hpf(A)$. Let us denote this connective $\K$-theory functor of small stable $\infty$-categories by $\bKc:\catex\functor\Sp$. 

\begin{lem} \label{connWald}
 There is a natural equivalence of spectra $\bK^w(\hpfw(A))\overset{\sim}{\map}\bKc(\hpf(A))$.
\end{lem}

\begin{proof}
 Since in the previous Lemma we showed that $\hpfw(A)$ is a Waldhausen subcategory of a model category, the assertion follows from Corollary 7.12 of \cite{BluGepTab}.
\end{proof}

\begin{rem}
 We obtain yet another description of $\KQ$-theory, viz., $$\KQ_i(A)\cong\pi_i(\bKc(\hpf(A))).$$
\end{rem}

Using the delooping machinery of \cite{SchNegK} one can define the nonconnective $\K$-theory spectrum. This task was carried out in Section 9 of \cite{BluGepTab} in the setting of stable $\infty$-categories. Let us denote the nonconnective $\K$-theory functor by $\bKn:\catex\functor\Sp$. 

\begin{prop}
 Let $0\map A\map B\map C\map 0$ be a short exact sequence with $B,C$ unital, $A^2=A$, and $A$ a flat $B$-module. Then there is a cofiber sequence in $\Sp$ $$\bKn(\hpf(A))\map\bKn(\hpf(B))\map\bKn(\hpf(C)).$$
\end{prop}

\begin{proof}
It follows from Lemma \ref{homfib} that $\hpf({A})\map\hpf(B)\map\hpf(C)$ is exact in $\catex$. The assertion now follows since nonconnective algebraic $\K$-theory satisfies localization (see Theorem 9.8 of \cite{BluGepTab}).
\end{proof}

In \cite{BluGepTab} the authors constructed the univeral localizing invariant $\uloc: \catex\functor \mloc$ and proposed $\mloc$ as a candidate for noncommutative motives in the setting of stable $\infty$-categories (see Theorem 8.7 of \cite{BluGepTab}). We set $\cM_\infty(A):=\uloc\circ\hpf(A):\Alg_k\functor\mloc$ and call it the {\em stable $\infty$-category valued noncommutative motive of $A$}. In fact, the $\infty$-category $\mloc$ is itself stable and the exact sequences in $\catex$ produce cofiber sequences in $\mloc$. It follows from Lemma \ref{homfib} that 

\begin{lem} \label{Mexc}
For any short exact sequence in $0\map A\map B\map C\map 0$ with $B,C$ unital, $A^2=A$, and $A$ a flat $B$-module, there is cofiber sequence in $\mloc$ $$\cM_\infty(A)\map\cM_\infty(B)\map\cM_\infty(C).$$
\end{lem}

For any $A\in\Alg_k$ let $M_n(A)$ denote the $k$-algebra of $n\times n$-matrices over $A$. There is a canonical corner embedding $A\map M_n(A)$ sending $a\mapsto a(e_{11})$.

\begin{prop}
 For any $A\in\Alg_k$ there is an equivalence $\cM_\infty(A)\simeq\cM_\infty(M_n(A))$ induced by the corner embedding $A\map M_n(A)$. 
\end{prop}

\begin{proof}
 From Proposition 2.4 of \cite{KQ} we deduce that $\hpfd(A)\cong\hpfd(M_n(A))$ induced by the corner embedding $A\map M_n(A)$. Consequently, $\hpf(A)\simeq \hpf(M_n(A))$ in $\catex$ (see Proposition \ref{infFun}). The assertion follows since $\cM_\infty(-)=\uloc\circ\hpf(-)$.
\end{proof}

Let $\CAlg$ denote the category of all (possibly nonseparable) $C^*$-algebras viewed as a subcategory of $\Alg_\CC$. It follows from the Cohen--Hewitt factorization theorem that any $A\in\CAlg$ satisfies $A^2=A$. Summarizing, we have the following:

\begin{thm} \label{mloc}
 Viewing $\CAlg$ as an ordinary category (not a topological category) there is a functor $\cM_\infty:\N(\CAlg)\functor\mloc$ that satisfies: 
 
 \begin{enumerate}
  \item (matrix stability): $\cM_\infty(A)\simeq\cM_\infty(M_n(A))$ for all $A$, and
  \item (localization / excision): any short exact sequence $0\map A\map B\map C\map 0$ produces the following cofiber sequence in $\mloc$ $$\cM_\infty(A)\map\cM_\infty(B)\map\cM_\infty(C).$$
 \end{enumerate}

\end{thm}

\begin{proof}
 Only $(2)$ needs a proof because it has been strengthened (note that $B$ and $C$ are no longer assumed to be unital). For any short exact sequence $0\map A\map B\overset{g}{\map} C\map 0$ in $\CAlg$ we need to show that $\cM_\infty(A)$ is the fiber of the induced map $\cM_\infty(B)\overset{g}{\map}\cM_\infty(C)$. We can form another short exact sequence $0\map A\map \tilde{B}\overset{\tilde{g}}{\map} \tilde{C}\map 0$ with $\tilde{B},\tilde{C}$ unital. Now there is a diagram in $\mloc$ (see Example \ref{homEpi} and Lemma \ref{Mexc})
 
 \beqn
 \xymatrix{
 \mathrm{fib}(g)\ar[r]\ar[d] & \cM_\infty(B)\ar[r]^g\ar[d] & \cM_\infty(C)\ar[d] \\
  \cM_\infty(A)\ar[r]\ar[d] & \cM_\infty(\tilde{B})\ar[r]^{\tilde{g}}\ar[d] & \cM_\infty(\tilde{C})\ar[d] \\
  0\simeq\cM_\infty(0)\ar[r] & \cM_\infty(\CC)\ar[r]^\cong\ar@/_1pc/[u] & \cM_\infty(\CC).\ar@/_1pc/[u]
 }
 \eeqn The bottom two rows and the two columns on the right are cofiber sequences. Since $0\map B\map\tilde{B}\map \CC\map 0$ and $0\map C\map\tilde{C}\map \CC\map 0$ admit splittings $\CC\map\tilde{B}$ and $\CC\map\tilde{C}$ in $\CAlg$ respectively, the indicated splittings exist in the above diagram, i.e., the two columns on the right are split cofiber sequences. The assertion now follows by a diagram chase.
\end{proof}

\subsection{$C^*$-algebras and nonconnective $\KQ$-theory}
We begin with an alternative description of the $\infty$-category of pointed noncommutative spaces $\iNS$. Let $\Csep^\delta$ temporarily denote the category of separable $C^*$-algebras, where the morphism spaces carry the discrete topology; moreover, let $\Csep$ denote the topological category of separable $C^*$-algebras, where the morphism spaces carry the point-norm topology. We denote the topological nerves of these categories by $\iCsep^\delta$ and $\iCsep$ respectively. There is a canonical functor $\iCsep^\delta\overset{\gamma}{\functor}\iCsep$ that is identity on objects. Let $H=\{f\,|\, \text{$f$ homotopy equivalence in $\Csep^\op$}\}$ denote a collection of maps in $(\iCsep^\delta)^\op$. Setting $\cC:=(\iCsep^\delta)^\op$ we find that $\gamma^\op:\cC\functor\iCsep^\op$ sends the maps in $H$ to equivalences. Thus it follows that $\gamma^\op$ factors as $$\cC\overset{}{\functor}\cC[H^{-1}]\overset{\theta}{\functor}\iCsep^\op.$$ Here the $\infty$-category $\cC[H^{-1}]$ is constructed as a fibrant replacement of the marked simplicial set $(\cC,H)$ (see Section 5.2.7 of \cite{LurToposBook}, also \cite{HinLoc}). The functor $\theta$ extends to a continuous functor $\Ind_\omega(\cC[H^{-1}])\overset{\theta}{\functor}\Ind_\omega(\iCsep^\op)=\iNS$. 

\begin{lem} \label{NS}
 The functor $\Ind_\omega(\cC[H^{-1}])\overset{\theta}{\functor}\iNS$ is an equivalence of $\infty$-categories.
\end{lem}

\begin{proof}
 The functor $\cC[H^{-1}]\overset{\theta}{\functor}\iCsep^\op$ is an equivalence (see Proposition 3.17 of \cite{BarJoaMe}), from which the assertion follows (see Proposition 5.3.5.11 of \cite{LurToposBook}).
\end{proof}

Let $\cM_\infty^\cpt$ denote the composite functor $\iCsep\overset{L_\cpt}{\functor}\iCsep\overset{\cM_\infty}{\functor}\mloc$, where the localization $L_\cpt$ is described in Example \ref{ConE}. Since $\mloc$ is a stable $\infty$-category its homotopy category $\h\mloc$ is triangulated and we denote the triangulated category valued functor $\h\cM_\infty^\cpt:\h\iCsep\functor\h\mloc$ simply by $\cM^\cpt$. In \cite{MyColoc} we constructed a stable presentable $\infty$-category $\pNSp[\cpt^{-1}]$, whose opposite $\infty$-category is by definition the bivariant $\K$-theory $\infty$-category $\ikk$ for arbitrary pointed noncommutative spaces, i.e., $\ikk:=\pNSp[\cpt^{-1}]^\op$. Moreover, it is shown in Theorem 2.4 of \cite{MyColoc} that there is a fully faithful exact functor $\KKcat\hookrightarrow\h\ikk$.

\begin{thm} \label{facKK}
 The functor $(\cM_\infty^\cpt)^\op:\iCsep^\op\functor\mloc^\op$ induces the following two functors: \begin{enumerate}
                                                                                \item $\iNS\functor\mloc^\op$ that is continuous, and
                                                                                \item $\ikk\functor\mloc$ that is exact.
                                                                               \end{enumerate}
\end{thm}

\begin{proof}
For (1) observe that the functor $\cM_\infty^\cpt$ satisfies localization / excision whence the functor $\cM^\cpt$ is split exact. Since $\cM$ is $M_n$-stable, $\cM^\cpt$ is also $C^*$-stable (see Proposition 3.31 of \cite{CunMeyRos}). By Higson's Theorem (see Theorem 3.2.2 of \cite{Hig2}) the functor $(\cM_\infty^\cpt)^\op:\cC=(\iCsep^\delta)^\op\functor\mloc^\op$ sends the maps in $H$ to equivalences and hence it factors as $\cC\functor\cC[H^{-1}]\overset{\Theta}{\functor}\mloc^\op.$ Since $\mloc^\op$ admits all (filtered) colimits, the functor $\cC[H^{-1}]\overset{\Theta}{\functor}\mloc^\op$ can be extended to a continuous functor $\Ind_\omega(\cC[H^{-1}])\overset{\Theta}{\functor}\mloc^\op$ and by Lemma \ref{NS} we may identify $\Ind_\omega(\cC[H^{-1}])\simeq\iNS$.

For (2) observe that the functor $\cC[H^{-1}]\overset{\Theta}{\functor}\mloc^\op$ preserves finite colimits. Indeed, thanks to (the dual of) Corollary 4.4.2.5 of \cite{LurToposBook} one simply needs to verify that the functor $\Theta^\op = \cM_\infty^\cpt$ preserves zero objects and pullbacks. Let $\Sp^\Sigma(\cC[H^{-1}])$ denote the filtered colimit of $\cC[H^{-1}]\overset{\Sigma}{\map}\cC[H^{-1}]\overset{\Sigma}{\map}\cC[H^{-1}]\overset{\Sigma}{\map}\cdots$ in $\catinf$. Since $\mloc^\op$ is stable $\Theta$ induces an exact functor $\Sp^\Sigma(\cC[H^{-1}])\overset{\Theta}{\map}\mloc^\op$ that, owing to the cocompleteness of $\mloc^\op$, can be extended to a functor $\Ind_\omega(\Sp^\Sigma(\cC[H^{-1}]))\overset{\Theta}{\map}\mloc^\op$. We may identify $\Ind_\omega(\Sp^\Sigma(\cC[H^{-1}]))\simeq\Sp(\Ind_\omega(\cC[H^{-1}]))\simeq\Sp(\iNS)$ furnishing $\Sp(\iNS)\overset{\Theta}{\functor}\mloc^\op$. By construction in \cite{MyColoc} $\pNSp$ is an accessible localization of $\Sp(\iNS)$ and $\pNSp[\cpt^{-1}]$ is a smashing colocalization $\pNSp$ whence there are fully faithful exact functors $\pNSp[\cpt^{-1}]\hookrightarrow \pNSp\hookrightarrow \Sp(\iNS)$. Its composition with $\Sp(\iNS)\overset{\Theta}{\functor}\mloc^\op$ is also exact, whose opposite produces the desired exact functor $\ikk\functor\mloc$.
\end{proof}

\begin{rem} \label{BivHomColoc}
 It follows from Theorem \ref{facKK} that the functor $\cM^\cpt$ is a bivariant homology theory. Using results of \cite{MyComparison} one can also show that $\cM^{\Oinf}(-)=\cM(-\prot\Oinf)$ is a bivariant homology theory. Thus $\cM^\cpt,\cM^{\Oinf}:\Csep\functor\h\mloc$ are noncommutative motive valued bivariant homology theories on the category of separable $C^*$-algebras.
\end{rem}

\begin{defn}
 We define the nonconnective $\KQ$-theory or $\KQ^\nc$-theory groups as $$\KQ^\nc_i(A):=\pi_i(\bKn(\hpf(A))) \text{ for all $A\in\Alg_k$ and $i\in\ZZ$.}$$
\end{defn}

\begin{thm} \label{KQ}
 Let a $C^*$ algebra $B$ be of the form $A\prot C$, where $C=\cpt$ or any properly infinite $C^*$-algebra. Then there is a natural isomorphism $\KQ^\nc_i(B)\cong\K^\nc_i(B)$ for all $i\in\ZZ$, where $\K^\nc_i(B)$ denotes the $i$-th nonconnective algebraic $\K$-theory group of $B$.
\end{thm}

\begin{proof}
 Let us first address the case where $C=\cpt$ and to this end we set $A_\cpt = A\prot\cpt$. From Lemma \ref{Waldstr} we have a canonical Waldhausen exact functor $\hpfw(A_\cpt)\map\Perf(\tilde{A_\cpt})$. The composite $\hpfw(A_\cpt)\map\Perf(\tilde{A_\cpt})\overset{\phi}{\map}\Perf(\CC)$ is trivial, where $\phi=-\otimes^{\mathbb{L}}_{\tilde{A_\cpt}} \CC$. It follows that there is a canonical map of stable $\infty$-categories $\hpf(A_\cpt)\map F(\phi)$, where $F(\phi)$ is the fiber of the map $\N(\cM(\Perf(\tilde{A_\cpt}))^\mathrm{cf})\overset{\phi}{\map}\N(\cM(\Perf(\CC))^\mathrm{cf})$ (see Lemma 7.11 of \cite{BluGepTab} for the construction of $\cM(\Perf(-))$, which is a Waldhausen subcategory of a simplicial model category. Using Lemma \ref{connWald} and Theorem 3.7 of \cite{KQ} we deduce that the map $\hpf(A_\cpt)\map F(\phi)$ is a connective $\K$-theory isomorphism, i.e., there is an equivalence  $\bKc(\hpf(A_\cpt))\overset{\sim}{\map}\bKc(F(\phi))$. The connective $\K$-theory spectrum of $F(\phi)$ can be identified with $\bKc(A_\cpt)$, i.e., the connective algebraic $\K$-theory spectrum of $A_\cpt$ due to excision \cite{SusWod2}. Note that the nonconnective algebraic $\K$-theory spectrum of a stable $\infty$-category is defined in such a manner (see Section 9 of \cite{BluGepTab}) so that when applied to a $C^*$-algebra $B$ it produces the expected result, viz., $$\bKn(\hpf(B)) \simeq\colim_n\, \Omega^n\bKc(\Sigma_\kappa^{(n)}\hpf(B))\simeq \colim_n\, \Omega^n \bKc (\Sigma^n B),$$ where $\Sigma^n B$ denotes the $n$-th Karoubi delooping of $B$ (see, for instance, \cite{SchSedano,friendlyMatch}). Using localization the nonconnective $\K$-theory spectrum of $F(\phi)$ can be identified with $\bKn(A_\cpt)$ \cite{BluGepTab}, which proves the assertion for stable $C^*$-algebras.

If $C$ is properly infinite then using Proposition 2.2 of \cite{CorPhi} (see also \cite{ThomThesis}) one obtains a commutative diagram in $\CAlg$

\beq
\xymatrix{
C\ar[rr]^\iota\ar[rd]_\theta && M_2(C) \\
& C\prot\cpt \ar[ur]_\kappa,
}
\eeq where the top horizontal arrow $\iota:C\map M_2(C)$ is the corner embedding. Tensoring the above diagram with a unital $A$ and applying the functors $\KQ^\nc(-)$ and $\K^\nc(-)$ along with the natural transformation between them produces a commutative diagram

\beqn
\xymatrix{
\KQ^\nc_m(A\prot C)\ar[r]\ar[d] & \KQ^\nc_m(A\prot C\prot\cpt) \ar[r]\ar[d]^\cong & \KQ^\nc_m(M_2(A\prot C))\ar[d] \\
\K^\nc_m(A\prot C)\ar[r] & \K^\nc_m(A\prot C\prot\cpt) \ar[r] & \K^\nc_m (M_2(A\prot C)),}
\eeqn where the middle verticle arrow is an isomorphism (since $A\prot C\prot\cpt$ is stable). Observe that both $\KQ^\nc$-theory and $\K^\nc$-theory are matrix stable whence the top and the bottom horizontal compositions are isomorphisms. The assertion in the unital case now follows by a diagram chase. Finally using excision one can prove the general case.
\end{proof}

\begin{rem}
 The argument above actually shows that there is a map of spectra that induces the isomorphism at the level of homotopy groups, which are the $\KQ^\nc$-theory and $\K^\nc$-theory groups in the source and target respectively. The map of connective spectra can also be delooped inductively by a Bass--Heller--Swan splitting argument \cite{RosComparison}.
\end{rem}

\begin{rem} \label{identical}
Observe that $\Oinf$ is properly infinite whence the above Theorem \ref{KQ} is applicable to $\Oinf$-stable $C^*$-algebras. Since we already know that $\K^\nc$-theory of a stable or an $\Oinf$-stable $C^*$-algebra agrees naturally with its topological $\K$-theory (see \cite{SusWod2,CorPhi,MyComparison}), we conclude that $\KQ^\nc$-theory is naturally isomorphic to topological $\K$-theory for such a $C^*$-algebra. From the computational viewpoint it turns out that for such a $C^*$-algebra
\beqn
\text{connective $\E$-theory $\cong$ $\KQ^\nc$-theory $\cong$ $\K^\nc$-theory $\cong$ topological $\K$-theory.}
\eeqn Let us also remark that topological $\K$-theory is Bott $2$-periodic and fairly easy to compute.
\end{rem}

\begin{rem} \label{categorification}
 The above Theorems \ref{facKK} and \ref{KQ} are the key ingredients in the categorification of topological $\TT$-duality. Intuitively, our result asserts that under favourable circumstances topological $\TT$-duality induces an equivalence of noncommutative motives associated with certain $C^*$-algebras (or, more generally, noncommutative spaces). Upon passing to the nonconnective $\KQ$-theory one recovers the familiar twisted $\K$-theory isomorphism. For the details we refer the readers to (Example 4.1 of \cite{KQ} and Section 1 of \cite{MyComparison}).
\end{rem}

%----------------------------------------bibliography------------------------------------------------------------

\bibliographystyle{abbrv}

\bibliography{/home/ibatu/Professional/math/MasterBib/bibliography}

\smallskip

\end{document}